\documentclass[11pt]{article}
\usepackage{latexsym,amsfonts,amssymb,amsmath,amsthm}
\usepackage{graphicx}

\usepackage[usenames,dvipsnames]{color}
\usepackage{ulem}

\begin{document}
\setlength{\baselineskip}{16pt}

\parindent 0.5cm
\evensidemargin 0cm \oddsidemargin 0cm \topmargin 0cm \textheight 22cm \textwidth 16cm \footskip 2cm \headsep
0cm

\newtheorem{theorem}{Theorem}[section]
\newtheorem{lemma}{Lemma}[section]
\newtheorem{proposition}{Proposition}[section]
\newtheorem{definition}{Definition}[section]
\newtheorem{example}{Example}[section]
\newtheorem{corollary}{Corollary}[section]

\newtheorem{remark}{Remark}[section]
\numberwithin{equation}{section}

\def\p{\partial}
\def\I{\textit}
\def\R{\mathbb R}
\def\C{\mathbb C}
\def\u{\underline}
\def\l{\lambda}
\def\a{\alpha}
\def\O{\Omega}
\def\e{\epsilon}
\def\ls{\lambda^*}
\def\D{\displaystyle}
\def\wyx{ \frac{w(y,t)}{w(x,t)}}
\def\imp{\Rightarrow}
\def\tE{\tilde E}
\def\tX{\tilde X}
\def\tH{\tilde H}
\def\tu{\tilde u}
\def\d{\mathcal D}
\def\aa{\mathcal A}
\def\DH{\mathcal D(\tH)}
\def\bE{\bar E}
\def\bH{\bar H}
\def\M{\mathcal M}
\renewcommand{\labelenumi}{(\arabic{enumi})}

\def\disp{\displaystyle}
\def\undertex#1{$\underline{\hbox{#1}}$}
\def\card{\mathop{\hbox{card}}}
\def\sgn{\mathop{\hbox{sgn}}}
\def\exp{\mathop{\hbox{exp}}}
\def\OFP{(\Omega,{\cal F},\PP)}
\newcommand\JM{Mierczy\'nski}
\newcommand\RR{\ensuremath{\mathbb{R}}}
\newcommand\CC{\ensuremath{\mathbb{C}}}
\newcommand\QQ{\ensuremath{\mathbb{Q}}}
\newcommand\ZZ{\ensuremath{\mathbb{Z}}}
\newcommand\NN{\ensuremath{\mathbb{N}}}
\newcommand\PP{\ensuremath{\mathbb{P}}}
\newcommand\abs[1]{\ensuremath{\lvert#1\rvert}}
\newcommand\normf[1]{\ensuremath{\lVert#1\rVert_{f}}}
\newcommand\normfRb[1]{\ensuremath{\lVert#1\rVert_{f,R_b}}}
\newcommand\normfRbone[1]{\ensuremath{\lVert#1\rVert_{f, R_{b_1}}}}
\newcommand\normfRbtwo[1]{\ensuremath{\lVert#1\rVert_{f,R_{b_2}}}}
\newcommand\normtwo[1]{\ensuremath{\lVert#1\rVert_{2}}}
\newcommand\norminfty[1]{\ensuremath{\lVert#1\rVert_{\infty}}}
\newcommand{\ds}{\displaystyle}

\title{Traveling Wave Solutions of Spatially Periodic Nonlocal Monostable Equations\thanks{Partially supported by NSF grant DMS--0907752}}

\author{
Wenxian Shen and Aijun Zhang \\
Department of Mathematics and Statistics\\
Auburn University\\
Auburn University, AL 36849\\
U.S.A. }

\date{}
\maketitle

\noindent {\bf Abstract.}
This paper deals with front propagation dynamics of monostable equations with nonlocal dispersal in spatially periodic habitats.
In the authors' earlier works, it is shown that a general spatially periodic monostable equation with nonlocal dispersal
 has a unique spatially periodic positive stationary solution  and has a spreading speed  in every direction.
In this paper, we show that a spatially periodic nonlocal monostable equation
with  certain spatial homogeneity or small nonlocal dispersal distance   has a unique stable  periodic traveling wave solutions
connecting its unique spatially periodic positive stationary solution and the trivial solution in every direction for all  speeds
greater than  the spreading speed in that direction.

\bigskip


\bigskip


\newpage

\section{Introduction}

The current paper is concerned with traveling wave solutions of spatially periodic nonlocal monostable equations.

Monostable  equations arise in modeling  population dynamics of species in biology and ecology.
Classically, one assumes that the internal interaction of species is random and local (i.e.
 species moves randomly between the adjacent spatial locations),
 which leads to the following reaction-diffusion equation,
 \begin{equation}
\label{random-eq} \frac{\p u}{\p t}= \Delta u+uf(x,u),\quad x\in\RR^N,
\end{equation}
where
$u(t,x)$ represents the population density of species at time $t$ and spatial location $x$ and
$f(x,u)$ satisfies certain monostablility assumptions. Roughly, the monostablility assumptions mean that
 $f(x,u)<0$ for $u\gg 1$, $f_u(x,u)<0$ for $u\geq 0$,
and the trivial solution $u=0$ is unstable.

In reality, the movements and interactions of many
species in biology and ecology can occur between non-adjacent spatial locations. Taking the nonlocal internal interaction
of species into the account  leads to the following nonlocal dispersal evolution equation,
\begin{equation}
\label{main-eq}
 \frac{\p u}{\p
 t}=\int_{\RR^N}k(y-x)u(t,y)dy-u(t,x)+u(t,x)f(x,u(t,x)),\quad
 x\in\RR^N,
\end{equation}
where
 $k(\cdot)$ is a $C^1$
 convolution kernel supported on a ball centered at the origin (that is, there is  a $\delta_0>0$ such that
  $k(z)>0$ if $\|z\|<\delta_0$, $k(z)=0$ if $\|z\|\geq \delta_0$, where
 $\|\cdot\|$ denotes the norm
 in $\RR^N$ and $\delta_0$ represents the nonlocal dispersal distance),  $\int_{\RR^N}k(z)dz=1$,
  and $f(x,u)$ satisfies certain monostable assumptions.

 Throughout this paper, we assume that $f(x,u)$ is periodic in $x$ with period vector $\mathbf{p}=(p_1,p_2,\cdots, p_N)$ (that is,  $f(\cdot+p_i{\bf e_i},\cdot)=f(\cdot,\cdot)$,
${\bf e_i}=(\delta_{i1},\delta_{i2}, \cdots,\delta_{iN})$, $\delta_{ij}=1$ if $i=j$ and $0$ if $i\not =j$,
$i,j=1,2,\cdots,N$).
    To state the monostablility assumptions on $f$, let
\begin{equation}
\label{x-p-space} X_p=\{u\in C(\RR^N,\RR)|u(\cdot+p_i{\bf e_i})=u(\cdot),\quad i=1,\cdots,N\}
\end{equation}
with norm $\|u\|_{X_p}=\ds\sup_{x\in\RR^N}|u(x)|$, and
\begin{equation}
\label{x-p-positive-space} X_p^+=\{u\in X_p\,|\, u(x)\geq 0\quad \forall x\in\RR^N\}.
\end{equation}
Let  $I$ be the identity map on $X_p$, and $\mathcal{K}$, $a_0(\cdot)I:X_p\to X_p$ be defined by
\begin{equation}
\label{k-delta-op} \big(\mathcal{K} u\big)(x)=\int_{\RR^N}k(y-x)u(y)dy,
\end{equation}
\begin{equation}
\label{a-op} (a_0(\cdot)Iu)(x)=a_0(x)u(x),
\end{equation}
where $a_0(x)=f(x,0)$.
We assume   the following  monostablility assumptions
for \eqref{random-eq}:

\medskip
 \noindent{\bf (A1)} {\it $f\in C^1(\RR^N\times [0,\infty),\RR)$,
 $\ds\sup_{x\in\RR^N,u\geq 0}\frac{\p f(x,u)}{\p u}<0$   and
 $f(x,u)<0$ for $x\in\RR^N$ and $u\gg 1$.}

 \medskip
\noindent{\bf (A2)} {\it $u\equiv 0$ is linearly unstable in $X_p$, that is,
the principal eigenvalue of
$$
\begin{cases}
\Delta u+a_0(x)u=\lambda u,\quad x\in\RR^N\cr
u(x+p_i\mathbf{e_i})=u(x),\,\,\, i=1,2,\cdots,N,\,\, x\in\RR^N
\end{cases}
$$
is positive. }

 \medskip

The following are monostablility assumptions for \eqref{main-eq}:

\medskip
\medskip
 \noindent{\bf (H1)}  {\it $f\in C^1(\RR^N\times [0,\infty),\RR)$,
 $\ds\sup_{x\in\RR^N,u\geq 0}\frac{\p f(x,u)}{\p u}<0$   and
 $f(x,u)<0$ for $x\in\RR^N$ and $u\gg 1$.}

 \medskip
\noindent{\bf (H2)} {\it $u\equiv 0$ is linearly unstable in $X_p$, that is, $\lambda_0:=\sup\{{\rm
Re}\lambda\,|\, \lambda\in\sigma(\mathcal{K}-I+a_0(\cdot)I)\}$ is positive, where  $\sigma(\mathcal{K}-I+a_0(\cdot)I)$
is the spectrum of the operator
 $\mathcal{K}-I+a_0(\cdot)I$ on $X_p$.}

 \medskip

It is well known that (A1) and (A2)  imply  that \eqref{random-eq} has exactly two
equilibrium solutions in $X_p^+$,  $u= 0$ and $u= u^+$, and $u= 0$ is linearly unstable and $u=u^+$ is
asymptotically stable in $X_p$, which reflects the monostable feature of the assumptions (A1) and (A2).

Observe that (A1) and (H1) are exactly the same (it is for convenience to state them separately). (H2) is the analogue of (A2). It should be pointed out that
$\lambda_0$ in (H2) may not be an eigenvalue of $\mathcal{K}-I+a_0(\cdot)I$ (see an example in
\cite{ShZh0}) and therefore there is some
essential difference between random dispersal and nonlocal dispersal operators.
Nevertheless, it is proved in  \cite{ShZh1}  that  (H1) and (H2) also imply  that  \eqref{main-eq} has exactly two
equilibrium solutions in $X_p^+$,  $u= 0$ and $u= u^+$, and $u= 0$ is linearly unstable and $u=u^+$ is
asymptotically stable in $X_p$ (see Proposition \ref{positive-solu-prop} or
\cite[Theorem C]{ShZh1}), which reflects the monostable feature of the assumptions (H1) and (H2).

Among the most important dynamical issues about monostable equations
\eqref{random-eq} and \eqref{main-eq} are spatial spread and front
propagation dynamics.

The spatial spread and front propagation dynamics of
\eqref{random-eq} has been extensively studied since the pioneering
works by Fisher \cite{Fisher} and Kolmogorov, Petrowsky, Piscunov
\cite{KPP} on the following special case of \eqref{random-eq}
\begin{equation}
\label{classical-fisher-eq}
 \frac{\p u}{\p t}=\frac{\p ^ 2u}{\p x^2}+u(1-u),\quad\quad x\in \RR,
\end{equation}
which models the evolutionary take-over of a habitat by a fitter
genotype. See, for example, \cite{ArWe2},
\cite{BeHaNa1}, \cite{BeHaNa2}, \cite{BeHaRo}, \cite{FiMc}, \cite{FrGa},
 \cite{Ham}, \cite{HuSh}, \cite{HuZi1}, \cite{Kam}, \cite{LiZh}, \cite{LiZh1},
\cite{LiYiZh}, \cite{Lui}, \cite{Nad}, \cite{NoRuXi}, \cite{NoXi1},
\cite{RyZl}, \cite{Sat}, \cite{She1}, \cite{She3}, \cite{She4},
\cite{Uch}, \cite{Wei1}, \cite{Wei2}, and references therein, for
the study of  the spatial spread and front propagation dynamics of \eqref{random-eq}.
It is proved  that,   if  (A1) and (A2) hold,
then
 for every $\xi\in S^{N-1}:=\{\xi\in\RR^N|\,\|\xi\|=1\}$,
  there is a $c^*(\xi)\in\RR$ such that for every $c\geq c^*(\xi)$, there is a
 traveling wave solution connecting $u^+$ and $u^-\equiv 0$ and propagating in the direction of $\xi$
 with speed $c$, and there is no such traveling
 wave solution of slower speed in the direction of $\xi$.  Moreover, the minimal wave speed $c^*(\xi)$ is
 of some  important spreading  properties (hence is also called the spreading speed in the direction of
 $\xi$) and has the following variational
characterization. Let $\lambda(\xi,\mu)$ be the eigenvalue  of
\begin{equation}
\label{laplacian-eigenvalue-eq}
\begin{cases}
\Delta u-2\mu\sum_{i=1}^N\xi_i\frac{\p u}{\p x_i}+(a_0(x)+\mu^2)u=\lambda u,\quad x\in\RR^N\cr
u(x+p_i{\bf e_i})=u(x),\quad i=1,2,\cdots,N\quad x\in\RR^N
\end{cases}
\end{equation}
with largest real part, where $a_0(x)=f(x,0)$ (it is well known that $\lambda(\xi,\mu)$ is real and
algebraically simple. $\lambda(\xi,\mu)$ is called the {\it principal eigenvalue} of
\eqref{laplacian-eigenvalue-eq} in literature).
 Then
\begin{equation}
\label{random-variational-eq} c^*(\xi)=\inf_{\mu>0}\frac{\lambda(\xi,\mu)}{\mu}.
\end{equation}
(See \cite{BeHaNa1}, \cite{BeHaNa2}, \cite{BeHaRo}, \cite{LiZh},
\cite{Nad},  \cite{NoRuXi}, \cite{NoXi1}, \cite{Wei2} and references
therein for the above mentioned properties).

Recently,  various dynamical problems  related  to the spatial spread and front propagation
 dynamics of
 nonlocal dispersal equations of the form \eqref{main-eq}   have also been studied by many authors. See, for example,
\cite{BaZh}, \cite{ChChRo}, \cite{CoElRo},  \cite{Cov},
\cite{CoDaMa1}, \cite{GaRo}, \cite{GrHiHuMiVi}, \cite{HeNgSh}, \cite{HeShZh},
\cite{HuGr}, \cite{HMMV}, \cite{HuShVi},
\cite{KaLoSh}, \cite{ShVi}, for the study of  spectral theory for
nonlocal dispersal operators and the existence, uniqueness, and
stability of nontrivial positive  stationary solutions. See, for
example, \cite{Cov1}, \cite{CoDu}, \cite{CoDaMa2}, \cite{LiSuWa},
\cite{LvWa}, \cite{PaLiLi}, \cite{Wei1}, \cite{Wei2}, for the study
of entire solutions and the existence of spreading speeds and
traveling wave solutions connecting the trivial solution $u=0$ and a
nontrivial positive stationary solution for some special cases of
\eqref{main-eq}. In particular, if $f(x,u)$ is independent of $x$,
then it is proved that \eqref{main-eq} has a spreading speed
$c^*(\xi)$ in every direction of $\xi\in S^{N-1}$
 ($c^*(\xi)$ is indeed independent of $\xi\in S^{N-1}$ in this case) and for every $c\geq c^*(\xi)$, \eqref{main-eq}
 has a traveling wave solution connecting $u^+$ and $0$ and propagating in the direction of $\xi$ with propagating speed $c$
 (see \cite{Cov1}).
 In the very recent papers \cite{ShZh0}, \cite{ShZh1}, the authors of the current paper explored the spatial spread
dynamics of general spatially periodic monostable equations and proved that assume (H1) and (H2),
\eqref{main-eq} has a spreading speed $c^*(\xi)$ in every direction of $\xi\in S^{N-1}$,
which extends the existing results
 on  spreading speed
of \eqref{random-eq} to
\eqref{main-eq}.

However, there is little understanding about  traveling wave solutions of spatially periodic monostable equations with nonlocal dispersal.
The objective of the current paper is to investigate the existence, uniqueness, and stability of traveling wave solutions of \eqref{main-eq}.
We show that if the periodic habitat of \eqref{main-eq} is of certain homogeneity or the nonlocal dispersal distance is small,
then \eqref{main-eq} has a unique stable traveling wave solution which connects $u^+$ and $0$ and propagates in a given
direction $\xi\in S^{N-1}$ for all speeds greater than the spreading speed in the direction of  $\xi$. The main tools employed in the
proofs of the existence, uniqueness, and stability of traveling wave solutions of \eqref{main-eq} include  sub- and super-solutions
and the principal eigenvalue theory for nonlocal dispersal operators which has recently been established  in
\cite{ShZh0}.

It should be pointed out that the spatial spread and front
propagation dynamics of spatially discrete monostable equations has
also been widely studied. The reader is referred to \cite{Che},
\cite{ChGu1}, \cite{ChGu2}, \cite{GuHa}, \cite{GuWu},  \cite{HuZi2},
\cite{She2}, \cite{WuZo}, \cite{ZiHaHu}.

The rest of this paper is organized as follows. In section 2, we introduce some standing notations and the definition of
spatially periodic traveling wave solutions and state the main results of the paper. In section 3, we present the comparison
principle for solutions of \eqref{main-eq} and some related linear equations with nonlocal dispersal and construct some
sub- and super-solutions of \eqref{main-eq} to be used in the proofs of the main results. The existence of traveling wave
solutions is investigated in section 4.
Section 5 concerns the uniqueness and continuity of traveling wave solutions.
In section 6, we show the stability  of traveling wave solutions.

\section{Notations, Definitions, and Main Results}

In this section, we introduce some standing notations and the
definition of spatially periodic traveling wave solutions, and state
the main results of the paper.

First of all,  let $X_p$ and $X_p^+$ be as in \eqref{x-p-space} and \eqref{x-p-positive-space}, respectively.
Let
\begin{equation}
\label{x-space} X=\{u\in C(\RR^N,\RR)\,|\, u\,\, \text{is uniformly continuous on}\,\,\RR^N\,\,{\rm and}\,\,
\sup_{x\in\RR^N}|u(x)|<\infty\}
\end{equation}
with norm $\|u\|_X=\ds\sup_{x\in\RR^N}|u(x)|$, and
\begin{equation}
\label{x-positive-space} X^+=\{u\in X\,|\, u(x)\geq 0\quad \forall x\in\RR^N\}.
\end{equation}
Let
\begin{equation}
\label{general-space}
\tilde X=\{u:\RR^N\to \RR\,|\, u\,\, \text{is Lebesgue measurable and bounded}\}
\end{equation}
endowed with the norm $\|u\|_{\tilde X}=\ds\sup_{x\in\RR^N}|u(x)|$
and
\begin{equation}
\label{general-positive-space}
\tilde X^+=\{u\in \tilde X\,|\, u(x)\geq 0\quad \forall x\in\RR^N\}.
\end{equation}
Observe that $X_p\subset X\subset \tilde X$.

Consider the shifted equations of \eqref{main-eq},
\begin{equation}
\label{main-shifted-eq}
 \frac{\p u}{\p
 t}=\int_{\RR^N}k(y-x)u(t,y)dy-u(t,x)+u(t,x)f(x+z,u(t,x)),\quad
 x\in\RR^N
\end{equation}
where $z\in\RR^N$. By general semigroup theory (see \cite{Hen} and
\cite{Paz}), for any $u_0\in \tilde X$ and $z\in\RR$,
\eqref{main-shifted-eq} has a unique (local) solution $u(t,\cdot)\in
\tilde X$ with $u(0,x)=u_0(x)$. Let $u(t,x;u_0,z)$ be the solution
of \eqref{main-shifted-eq} with $u(0,x;u_0,z)=u_0(x)$. Note that if
$u_0\in X_p$ (resp. $X$), then $u(t,\cdot;u_0,z)\in X_p$ (resp.
$X)$.
 If $u_0\in \tilde X^+$, then $u(t,x;u_0)$ exists for
all $t\geq 0$ (see Proposition \ref{comparison-prop}).

A measurable function
$u:\RR\times \RR^N$ is call an {\it entire solution} of \eqref{main-eq} if
$u(t,x)$ is differentiable in $t\in\RR$ and satisfies \eqref{main-eq} for all
$t\in\RR$ and $x\in\RR^N$.

\begin{proposition}
\label{positive-solu-prop}
Assume (H1)-(H2).  Then \eqref{main-eq} has a unique stationary solution
$u^+(\cdot)\in X_p^+\setminus\{0\}$. Moreover, $u=u^+(\cdot)$ is asymptotically stable with respect to
perturbations in $X_p^+\setminus\{0\}$ and for any $\xi\in S^{N-1}$,
any  $u_0\in \tilde X^+$,
$u_0(x)\geq \delta$ for all  $x\in\RR^N$ with $x\cdot\xi\le m$  for some $m\in\RR$ and $\delta>0$,
and any $\epsilon>0$, there are $T>0$ and $R\in\RR$ such that
$$\sup_{x,z\in\RR^N,x\cdot\xi\leq r}|u(T,x;u_0,z)-u^+(x+z)|<\epsilon\quad \forall r\le R.
$$
\end{proposition}

\begin{proof}
It follows from the arguments in \cite[Theorem C]{ShZh1} and \cite[Proposition 2.3]{ShZh0}.
\end{proof}

For given function $g:\RR\times \RR^N\times \RR^N\to\RR$,  $\xi\in S^{N-1}$, and $c,\alpha\in\RR$,
we define the following limit:
$$
\lim_{x\cdot\xi-ct\to \infty(-\infty)}g(t,x,z)=\alpha\,\,\text{uniformly
in}\,\, z\in\RR^N
$$
if and only if
$$ \lim_{r\to
\infty(-\infty)}\sup_{t\in\RR,x,z\in\RR^N,x\cdot\xi-ct\geq r(\leq
r)}|g(t,x,z)-\alpha|=0.
$$

\begin{definition}[Traveling wave solution]
\label{wave-def}
\begin{itemize}
\item[(1)]
An entire solution $u(t,x)$ of \eqref{main-eq} is called a {\rm traveling wave solution} connecting
$u^+(\cdot)$ and $0$ and propagating in the direction of $\xi$ with speed $c$ if there is a bounded
measurable function $\Phi:\RR^N\times \RR^N\to \RR^+$ such that $u(t,\cdot;\Phi(\cdot,z),z)$ exists
for all $t\in\RR$,
\begin{equation}
\label{wave-eq1} u(t,x)= u(t,x;\Phi(\cdot,0),0)=\Phi(x-ct\xi,ct\xi)\quad\forall t\in\RR,\,\, x\in\RR^N,
\end{equation}
\begin{equation}
\label{def-eq1}
u(t,x;\Phi(\cdot,z),z)=\Phi(x-ct\xi,z+ct\xi)\quad\forall t\in\RR,\,\, x,z\in\RR^N,
\end{equation}
\begin{equation}
\label{def-eq2}
\lim_{x\cdot\xi \to -\infty}\big(\Phi(x,z)-u^+(x+z)\big)=0,\quad \lim_{x\cdot\xi\to\infty}\Phi(x,z)=0\quad \text{uniformly in}\,\, z\in\RR^N,
\end{equation}
\begin{equation}
\label{def-eq3}
\Phi(x,z-x)=\Phi(x^{'},z-x^{'})\quad \forall x,x^{'}\in\RR^N\,\, \text{with}\,\, x\cdot\xi=x^{'}\cdot\xi,
\end{equation}
and
\begin{equation}
\label{def-eq4}
\Phi(x,z+p_i{\bf e_i})=\Phi(x,z)\quad \forall x,z\in\RR^N.
\end{equation}

\item[(2)] A bounded measurable function $\Phi:\RR^N\times\RR^N\to \RR^+$ is said to {\rm generate a traveling wave solution of
\eqref{main-eq} in the direction of $\xi$  with speed $c$} if it satisfies \eqref{def-eq1}-\eqref{def-eq4}.
\end{itemize}
\end{definition}

\begin{remark}
\label{wave-rk} Suppose that $u(t,x)=\Phi(x-ct\xi,c t\xi)$ is a
traveling wave solution of \eqref{main-eq} connecting $u^+(\cdot)$
and $0$ and propagating in the direction of $\xi$ with speed $c$.
Then $u(t,x)$ can be written as
\begin{equation}
\label{wave-eq2}
 u(t,x)=\Psi(x\cdot\xi-ct,x)
\end{equation}
for some $\Psi:\RR\times\RR^N\to \RR$ satisfying that
$\Psi(\eta,z+p_i{\bf e_i})=\Psi(\eta,z)$, $\ds\lim_{\eta\to -\infty} \Psi(\eta,z)=u^+(z)$, and
$\ds\lim_{\eta\to\infty}\Psi(\eta,z)=0$ uniformly in $z\in\RR^N$. In fact, let
$\Psi(\eta,z)=\Phi(x,z-x)$
for $x\in\RR^N$ with $x\cdot\xi=\eta$. Observe that $\Psi(\eta,z)$
is well defined and has the above mentioned properties. In some
literature, the form \eqref{wave-eq2} is adopted for spatially
periodic traveling wave solutions (see \cite{LiZh}, \cite{Nad},
\cite{Wei2}, and references therein).
\end{remark}

Next, we
recall some principal eigenvalue theory and spatial spreading theory established in \cite{ShZh0} and \cite{ShZh1}.

Consider the following eigenvalue problem, which
is a nonlocal counterpart of \eqref{laplacian-eigenvalue-eq},
\begin{equation}
\label{eigenvalue-eq} \big( \mathcal{K}_{\xi,\mu} -I +a(\cdot) I
\big)v=\lambda v,\quad v\in X_p,
\end{equation}
where $\xi\in S^{N-1}$, $\mu\in\RR$, and $a(\cdot)\in X_p$. The operator  $a(\cdot)I$  has the same
 meaning
  as in \eqref{a-op} with $a_0(\cdot)$ being replaced
 by $a(\cdot)$,
 and $\mathcal{K}_{\xi,\mu}:X_p\to X_p$ is defined by
\begin{equation}
\label{k-delta-xi-mu-op}
(\mathcal{K}_{\xi,\mu}v)(x)=\int_{\RR^N}e^{-\mu(y-x)\cdot\xi}k(y-x)v(y)dy.
\end{equation}
We point out the following relation between \eqref{main-eq} and
\eqref{eigenvalue-eq}: if $u(t,x)=e^{-\mu(
x\cdot\xi-\frac{\lambda}{\mu}t)}\phi(x)$ with $\phi\in
X_p\setminus\{0\}$ is a solution of
 the  linearization of \eqref{main-eq} at $u= 0$,
\begin{equation}
\label{linearization-eq0} \frac{\p u}{\p t}=\int_{\RR^N}
k(y-x)u(t,y)dy-u(t,x)+a_0(x)u(t,x),\quad x\in\RR^N,
\end{equation}
where $a_0(x)=f(x,0)$,  then $\lambda$ is an eigenvalue of
\eqref{eigenvalue-eq}  with $a(\cdot)=a_0(\cdot)$ or
$\mathcal{K}_{\xi,\mu}-I+a_0(\cdot)I$  and $v=\phi(x)$ is a
corresponding eigenfunction.

 Let
$\sigma( \mathcal{K}_{\xi,\mu}- I+a(\cdot)I)$ be the spectrum of
$ \mathcal{K}_{\xi,\mu}- I+a(\cdot)I$ on $X_p$. Let
$$\lambda_0(\xi,\mu,a):=\sup\{{\rm Re}\lambda\,|\,\lambda\in \sigma( \mathcal{K}_{\xi,\mu}- I+a(\cdot)I)\}.$$
We call $\lambda_0(\xi,\mu,a)$  the {\it principal spectrum point} of $\mathcal{K}_{\xi,\mu}- I+a(\cdot)I$.
Observe that if $\mu=0$, \eqref{eigenvalue-eq} is independent of $\xi$
and hence we put
\begin{equation}
\label{lambda-delta-a}
\lambda_0(a):=\lambda_0(\xi,0,a)\quad \forall\,\, \xi\in
S^{N-1}.
\end{equation}
$\lambda_0(\xi,\mu,a)$ is called the {\it principal eigenvalue} of $
\mathcal{K}_{\xi,\mu}- I+a(\cdot)I$ or $ \mathcal{K}_{\xi,\mu}- I+a(\cdot)I$ is said to have  a
principal eigenvalue if  $\lambda_0(\xi,\mu,a)$ is an algebraically
simple eigenvalue of $\mathcal{K}_{\xi,\mu}-I+a(\cdot)I$ with an eigenfunction
 $v\in X_p^+$,  and for every $\lambda\in
\sigma( \mathcal{K}_{\xi,\mu}- I+a(\cdot)I) \setminus
\{\lambda_0(\xi,\mu,a)\}$, ${\rm
Re}\lambda<\lambda_0(\xi,\mu,a)$.

Observe that $ \mathcal{K}_{\xi,\mu}- I+a(\cdot)I$ may not have a principal eigenvalue  (see an example in
\cite{ShZh0}), which reveals some essential difference between random dispersal operators and nonlocal dispersal operators.
The following  proposition on the existence of principal eigenvalue of $ \mathcal{K}_{\xi,\mu}- I+a(\cdot)I$ is proved in \cite{ShZh0}
(see also \cite{ShZh1}).

\begin{proposition}
\label{PE-prop}
\begin{itemize}
\item[(1)]
 If $k(x)= \frac{1}{\delta^N}\tilde k(\frac{x}{\delta})$ for all $x\in\RR^N$, where $\tilde k(\cdot)$ satisfies that $\tilde k(z)>0$ for
    $\|z\|<1$, $\tilde k(z)=0$ for $\|z\|\geq 1$, and $\int_{\RR^N}\tilde k(z)dz=1$, then
 $\lambda_0(\xi,\mu,a)$ is the principal eigenvalue
 of $\mathcal{K}_{\xi,\mu}-I+a(\cdot)I$
 for all  $\xi\in S^{N-1}$,
$\mu\in\RR$ and $0<\delta\ll 1$.

 \item[(2)] If $a(x)$ satisfies that $\ds\max_{x\in\RR^N}a(x)-\ds\min_{x\in\RR^N}a(x)<1$,
then  $\lambda_0(\xi,\mu,a)$ is the principal eigenvalue of $ \mathcal{K}_{\xi,\mu}- I+a(\cdot)I$
for  all $\xi\in S^{N-1}$ and $\mu\in\RR$.

\item[(3)]  If $a(\cdot)$ is $C^N$ and  the partial
derivatives of $a(x)$ up to order $N-1$ at some $x_0$ are zero, where $x_0$ is such that
$a(x_0)=\ds\max_{x\in\RR^N}a(x)$, then the conclusion in (2) holds.
\end{itemize}
\end{proposition}

Proposition \ref{PE-prop}  shows  such an important fact: nonlocal dispersal operator possesses a similar principal
eigenvalue theory to random dispersal operator for following cases:  the nonlocal dispersal is nearly local;   the periodic habitat
 is {\it nearly globally homogeneous}  (in the sense that the condition in Proposition \ref{PE-prop}(2) is satisfied)
 or it is {\it nearly homogeneous  in a region where it is most conducive
to population growth} in the zero-limit population (in the sense
that the condition in Proposition \ref{PE-prop}(3) is satisfied).
Note that if $a_0(\cdot)$ is $C^1$ and $1\leq N\leq 2$, the
condition in Proposition \ref{PE-prop}(3) is always satisfied.

As it is  mentioned above, a spatially periodic monostable equation with
random dispersal has a spreading speed in every direction. This
important feature has been well extended in \cite{ShZh0} and
\cite{ShZh1} to spatially periodic monostable equations with
nonlocal dispersal. For given function
$h:\RR\times\RR^N\times\RR^N$, we define
$$
\liminf_{x\cdot\xi\to -\infty}h(t,x,z)=\liminf_{r\to
-\infty}\inf_{x\in\RR^N,x\cdot\xi\leq r}h(t,x,z),
$$
$$
\limsup_{x\cdot\xi\to \infty}h(t,x,z)=\limsup_{r\to
\infty}\sup_{x\in\RR^N,x\cdot\xi\geq r}h(t,x,z),
$$
$$
\liminf_{t\to\infty}\inf_{x\cdot\xi\leq ct}h(t,x,z)=\liminf_{t\to \infty}\inf_{x\in\RR^N,x\cdot\xi\leq ct}h(t,x,z),
$$
and
$$
\limsup_{t\to\infty}\sup_{x\cdot\xi\geq ct}h(t,x,z)=\limsup_{t\to \infty}\sup_{x\in\RR^N,x\cdot\xi\geq ct}h(t,x,z).
$$
 Roughly speaking, a number
$c^*(\xi)\in\RR$ is called the {\rm spreading speed} of
\eqref{main-eq} in the direction of $\xi$ if for every $u_0\in X^+$
with $\ds\liminf_{x\cdot\xi \to -\infty}u_0(x)>0$ and $u_0(x)=0$ for
$x\cdot\xi\gg 1$,
$$
\liminf_{t\to\infty}\inf_{x\cdot\xi\leq
ct}(u(t,x;u_0)-u^+(x))=0\quad\forall c<c^*(\xi)
$$
and
$$
\limsup_{t\to\infty}\sup_{x\cdot\xi\geq ct}u(t,x;u_0)=0\quad \forall
c>c^*(\xi)
$$
(see \cite[Definition 1.2]{ShZh1} for detail).
The following proposition  on  the existence of spreading speeds is proved in \cite{ShZh1} (see also
\cite{ShZh0}).

\begin{proposition}
\label{spreading-prop}
Assume (H1) and (H2). For any $\xi\in S^{N-1}$, \eqref{main-eq} has a spreading
speed $c^*(\xi)$ in the direction of $\xi$. Moreover,  there is $\mu^*(\xi)>0$ such that
$$
c^*(\xi)=\inf_{\tilde\mu>0}\frac{\lambda_0(\xi,\tilde\mu,a_0)}{\tilde\mu}=\frac{\lambda_0(\xi,\mu^*(\xi),a_0)}{\mu^*(\xi)}
<\frac{\lambda_0(\xi,\mu,a_0)}{\mu}
\quad\forall\, \mu\in (0,\mu^*(\xi)).
$$
\end{proposition}

For convenience, we introduce the following standing assumption:

\medskip
\noindent {\bf (H3)}
 {\it
 For
every $\xi\in S^{N-1}$ and   $\mu\geq 0$,
$\lambda_0(\xi,\mu,a_0)$ is the principal eigenvalue of $\mathcal{K}_{\xi,\mu}-I+a_0(\cdot)I$,
 where $a_0(x)=f(x,0)$. }
\medskip

Biologically, one is only interested in nonnegative solutions of \eqref{main-eq}. Without loss of generality, we then also assume

\medskip
\noindent{\bf (H4)} {\it $f(x,u)=f(x,0)$ for $u\leq 0$.}
\medskip

We now state the main results of the paper.  For given $\xi\in S^{N-1}$ and $c>c^*(\xi)$, let $\mu\in (0,\mu^*(\xi))$ be such that
$$
c=\frac{\lambda_0(\xi,\mu,a_0)}{\mu}.
$$
If (H3) holds, let $\phi(\cdot)\in X_p^+$ be the positive principal eigenfunction of
$\mathcal{K}_{\xi,\mu}-I+a_0(\cdot)I$ with $\|\phi(\cdot)\|_{X_p}=1$.

\begin{theorem} [Existence of traveling wave solutions]
\label{existence-thm}
Assume (H1)-(H4). Then
 for any $\xi\in S^{N-1}$ and $c> c^*(\xi)$, there is a bounded measurable function $\Phi:\RR^N\times\RR^N\to \RR^+$ such that
 the following hold.
 \begin{itemize}
 \item[(1)] $\Phi(\cdot,\cdot)$ generates a traveling
wave solution  connecting $u^+(\cdot)$ and $0$ and propagating
in the direction of $\xi$ with speed $c$. Moreover,
$\ds\lim_{x\cdot\xi\to \infty}\frac{\Phi(x,z)}{e^{-\mu x\cdot\xi}\phi(x+z)}=1$
uniformly in $z\in\RR^N$.

\item[(2)]  Let  $U(t,x;z)=u(t,x;\Phi(\cdot,z),z)(=\Phi(x-ct\xi,z+ct\xi))$. Then
$$U_t(t,x;z)>0\quad \forall t\in\RR,\,\, x,z\in\RR^N,$$ $\ds\lim_{x\cdot\xi-ct\to -\infty}U_t(t,x;z)=0$, and
$\ds\lim_{x\cdot\xi-ct\to \infty}\frac{U_t(t,x;z)}{e^{-\mu(x\cdot\xi-ct)}\phi(x+z)}=\mu c$
uniformly in $z\in\RR^N$.
\end{itemize}
\end{theorem}

\begin{remark}
\label{wave-rk1} Let $\Phi(x,z)$ be as in Theorem \ref{existence-thm} and
 $\Psi(\eta,z)=\Phi(\eta\xi,z-\eta\xi)$.
Then $U(t,x;z)=\Psi(x\cdot\xi-ct,z+x)$ and $\Psi(\eta,z)$ is differentiable in $\eta$ and $\Psi_\eta(\eta,z)<0$.
\end{remark}

\begin{theorem}[Uniqueness and continuity of traveling wave solutions]
\label{uniqueness-thm}
Assume (H1)-(H4). Let  $\Phi(\cdot,\cdot)$ be  as in Theorem \ref{existence-thm}.
\begin{itemize}
\item[(1)]
Suppose that $\Phi_1(\cdot,\cdot)$ also generates a traveling wave solution of \eqref{main-eq}
in the direction of $\xi$ with speed $c$
  and
  $\ds\lim_{x\cdot\xi\to\infty}\frac{\Phi_1(x,z)}{\Phi(x,z)}=1\quad \text{uniformly in}\quad z\in\RR.$
Then $\Phi_1(x,z)\equiv \Phi(x,z).$

\item[(2)] $\Phi(x,z)$ is continuous in $(x,z)\in\RR^N$.
\end{itemize}
\end{theorem}

\begin{theorem}[Stability  of traveling wave solutions]
\label{stability-thm}
Assume (H1)-(H4). \\Let $U(t,x)=U(t,x;0)=\Phi(x-ct\xi,ct\xi)$, where $\Phi(\cdot,\cdot)$ is  as in Theorem \ref{existence-thm}.
For any $u_0\in X^+$ satisfying that  $\ds\lim_{x\cdot\xi\to \infty}\frac{u_0(x)}{U(0,x)}=1$
and $\ds\liminf_{x\cdot\xi\to -\infty}u_0(x)>0$, there holds
$$
\lim_{t\to\infty}\sup_{x\in\RR^N}\Big|\frac{u(t,x;u_0,0)}{U(t,x)}-1\Big|=0.
$$
\end{theorem}

We remark that by the spreading property of $c^*(\xi)$, it is not difficult to see that
\eqref{main-eq} has no traveling wave solutions in the direction of $\xi\in S^{N-1}$ with propagating
speed smaller than $c^*(\xi)$.
Theorems \ref{existence-thm}-\ref{stability-thm}  show the existence, uniqueness, and stability of traveling wave solutions of \eqref{main-eq} in
any given direction with speed greater than the spreading speed in that direction for
the above mentioned three special but important cases, that is,
the nonlocal dispersal is nearly local;   the periodic habitat
 is nearly globally homogeneous  or it is nearly homogeneous  in a region where it is most conducive
to population growth in the zero-limit population.
It remains open whether \eqref{main-eq} has a traveling wave solution in the given direction of $\xi\in S^{N-1}$ with
speed $c=c^*(\xi)$ for these special cases.
It also remains open whether a general spatially periodic monostable equation with nonlocal dispersal
in  $\RR^N$ with $N\geq 3$ has traveling wave solutions connecting the spatially periodic positive stationary  solution $u^+$ and $0$
and propagating with constant speeds.

\section{Comparison Principle and Sub- and Super-solutions}

In this section, we first in \ref{subsection-comparison} present the comparison principle for (sub-, super-) solutions of
\eqref{main-shifted-eq} and some related nonlocal linear evolution equations. Then we construct in \ref{subsection-sub-super-solu} some
sub- and super-solutions to be used in the proofs of the main results in later sections.

\subsection{Comparison principle}
\label{subsection-comparison}

Consider \eqref{main-shifted-eq}.
For given $a(\cdot,\cdot)\in C(\RR\times\RR^N,\RR)$ with $a(t,\cdot)\in X_p$ for every $t\in\RR$, consider also
\begin{equation}
\label{linear-new-eq1}
 \frac{\p u}{\p
 t}=\int_{\RR^N}k(y-x)u(t,y)dy-u(t,x)+ a(t, x+z)u(t,x),\quad
 x\in\RR^N.
\end{equation}

\begin{definition}
A  bounded Lebesgue measurable function $u(t,x)$ on $[0,T)\times \RR^N$ is called a {\it super-solution} (or {\it sub-solution}) of
\eqref{main-shifted-eq} if for any $x \in \RR^N$, u(t,x) is absolutely continuous on $[0,T)$(and so $\frac{\p u}{\p t}$ exists a.e on [0,T))  and satisfies
that for each $x\in\RR^N$,
$$
\frac{\p u}{\p t}\geq(or \leq) \int_{\RR^N}k(y-x)u(t,y)dy-u(t,x)+f(x+z,u)u(t,x)
$$
for a.e.  $t\in (0,T)$.
\end{definition}

Sub and super-solutions of \eqref{linear-new-eq1} are defined similarly.
Throughout this subsection, we assume (H1) and (H2).

\begin{proposition}[Comparison principle]
\label{comparison-prop} $\quad$
\begin{itemize}
\item[(1)]
If $u_1(t,x)$ and $u_2(t,x)$ are sub-solution and super-solution of \eqref{linear-new-eq1} on $[0,T)$,
respectively, $u_1(0,\cdot)\leq u_2(0,\cdot)$,  and $u_2(t,x)-u_1(t,x)\geq -\beta_0$ for $(t,x)\in [0,T)\times
\RR^N$ and some $\beta_0>0$, then
$u_1(t,\cdot)\leq u_2(t,\cdot)\quad {\rm for}\quad t\in [0,T).$

\item[(2)] If $u_1(t,x)$ and $u_2(t,x)$ are bounded sub- and super-solutions of \eqref{main-shifted-eq} on $[0,T)$,
respectively, and
$u_1(0,\cdot)\leq u_2(0,\cdot)$, then $u_1(t,\cdot)\leq u_2(t,\cdot)$ for $t\in[0,T)$.

\item[(3)] For every $u_0\in \tilde X^+$, $u(t,x;u_0,z)$ exists for all $t\geq 0$, where $u(t,x;u_0,z)$ is the solution
of \eqref{main-shifted-eq}
with $u(0,x;u_0,z)=u_0(z)$.

\item[(4)]  Suppose that $u_1,u_2\in \tilde X$, $u_1\leq u_2$, and $\{x\in\RR^N\,|\, u_2(x)>u_1(x)\}$ has positive Lebesgue measure. Then
 $u(t,x;u_1,z)< u(t,x;u_2,z)$ for every $t>0$ at which both $u(t,\cdot;u_1)$ and
$u(t,\cdot;u_2)$ exist and $x,z\in\RR^N$.
\end{itemize}
\end{proposition}

\begin{proof}
If follows from the arguments in \cite[Proposition 2.1]{ShZh0} and \cite[Proposition 2.2]{ShZh0}.
\end{proof}

\subsection{ Sub- and super-solutions}
\label{subsection-sub-super-solu}

 Throughout this subsection, we assume (H1)-(H4) and put $a_0(x)=f(x,0)$.

For given $\xi\in S^{N-1}$, let $\mu^*(\xi)$ be such that
$$
c^*(\xi)=\frac{\lambda_0(\xi,\mu^*(\xi),a_0)}{\mu^*(\xi)}.
$$
Fix $\xi\in S^{N-1}$ and $c>c^*(\xi)$. Let $0<\mu<\mu_1<\min\{2\mu,\mu^*(\xi)\}$ be such that
$
c=\frac{\lambda_0(\xi,\mu,a_0)}{\mu}
$
 and
$
\frac{\lambda_0(\xi,\mu,a_0)}{\mu}>\frac{\lambda_0(\xi,\mu_1,a_0)}{\mu_1}>c^*(\xi).
$
Let $\phi(\cdot)$ and $\phi_1(\cdot)$ be  positive eigenfunctions of $\mathcal{K}_{\xi,\mu}-I+a_0(\cdot)I$
associated to $\lambda_0(\xi,\mu,a_0)$ and $\lambda_0(\xi,\mu_1,a_0)$ with
$\|\phi(\cdot)\|_{X_p}=1$ and $\|\phi_1(\cdot)\|_{X_p}=1$, respectively.
If no confusion occurs, we may write $\lambda_0(\mu,\xi,a_0)$ as $\lambda(\mu)$.

For given $d_1>0$, let
\begin{equation}
\label{v-underbar-eq1}
\underline v^1(t,x;z,T,d_1)= e^{-\mu (x\cdot\xi+cT-ct)}\phi(x+z)-d_1
e^{-\mu_1(x\cdot\xi+cT-ct)}\phi_1(x+z).
\end{equation}
We may write $\underline v^1(t,x;z,T)$ for $\underline v^1(t,x;z,T,d_1)$ for fixed
$d_1>0$ or if no confusion occurs.

\begin{proposition}
\label{sub-solution-prop1}
 For  any $z\in\RR^N$ and $T>0$,
  $\underline v^1(t,x;z,T)$ is a sub-solution of
  \eqref{main-shifted-eq} provided that $d_1$ is sufficiently large.
\end{proposition}

\begin{proof} First of all,
 let $\varphi=e^{-\mu
(x\cdot\xi+cT-ct)}\phi(x+z)$ and
$\varphi_1=d_1e^{-\mu_1(x\cdot\xi+cT-ct)}\phi_1(x+z)$.
Let $M=
\ds\max_{x \in \RR^N}\phi(x)(>0)$. Let $L>0$ be such that
 $-f_{u}(x+z,u)\leq L$ for $0\leq u\leq M$.
 Let $d_0$ be defined by
 $$
d_0= \max\{\frac{\ds\max_{x\in\RR^N}\phi(x)}{\ds\min_{x\in\RR^N}\phi_{1}(x)}, \frac{L\ds\max_{x\in\RR^N}\phi^{2}(x)}{(\mu_{1} c -\lambda(\mu_{1}))\ds\min_{x\in\RR^N}\phi_{1}(x)}\}
$$

Fix $z\in\RR^N$ and $T>0$.
We prove that $\underline v^1(t,x;z,T)$ is a sub-solution of \eqref{main-shifted-eq}
for $d_1\geq d_0$, that is, for any $(t,x)\in\RR\times\RR^N$,
\begin{align}
\label{sub-solu-eq1}
\frac{\p \underline{v}^1}{\p t}- [\int_{\RR^N}k(y-x)\underline{v}^1(t,y;z,T)dy-\underline{v}^1(t,x;z,T)
+f(x+z,\underline{v}^1(t,x;z,T))\underline{v}^1 (t,x;z,T)]\leq 0.
\end{align}

First, for $(t,x)\in\RR\times\RR^N$ with $\underline
v^1(t,x;z,T)\leq 0$, $f(x+z,\underline v^1(t,x;z,T))=f(x+z,0)$. Hence
\begin{align*}
&\frac{\p \underline{v}^1}{\p t}- [\int_{\RR^N}k(y-x)\underline{v}^1(t,y;z,T)dy-\underline{v}^1(t,x;z,T)+f(x+z,\underline{v}^1(t,x;z,T))\underline{v}^1 (t,x;z,T)]\\
&= -(\mu_{1} c -\lambda(\mu_{1})) \varphi_1\leq 0.
\end{align*}
Therefore \eqref{sub-solu-eq1} holds for $(t,x)\in\RR\times\RR^N$ with $\underline
v^1(t,x;z,T)\leq 0$.

Next, consider $(t,x)\in\RR\times\RR^N$ with $\underline
v^1(t,x;z,T)>0$.  By $d_1\geq d_0$, we must have $x\cdot\xi+cT-ct \geq 0$. Then
$\underline{v}^1(t,x;z,T)\leq  e^{-\mu (x\cdot\xi+cT-ct)}\phi(x+z)
\leq  \phi(x+z) \leq M$. Note that for $0<y<M$,
\begin{align*}
-(\mu_{1} c -\lambda(\mu_{1}))-f_{u}(x+z,y)\frac{(\varphi)^{2}}{\varphi_1} &\leq -(\mu_{1} c -\lambda(\mu_{1}))+L \frac{(\varphi)^{2}}{\varphi_1}\\
&= -(\mu_{1} c -\lambda(\mu_{1}))+ \frac{L\phi^{2}(x+z)}{d_{1}\phi_{1}(x+z)} e^{(\mu_{1}-2\mu)(x\cdot\xi+cT-ct)}\\
 &\leq -(\mu_{1} c -\lambda(\mu_{1}))+ \frac{L\ds\max_{y \in \RR^N}\phi^{2}(y)}{d_{1}\ds\max_{y \in \RR^N}\phi_{1}(y)}\\
 &\leq 0.
\end{align*}
Therefore, for $(t,x)\in\RR\times\RR^N$ with $\underline v^1(t,x;z,T)>0$,
\begin{align*}
&\frac{\p \underline{v}^1}{\p t}- [\int_{\RR^N}k(y-x)\underline{v}^1(t,y;z,T)dy-\underline{v}^1(t,x;z,T)+f(x+z,\underline{v}^1)\underline{v}^1 (t,x;z,T)]\\
=&\mu c \varphi-\mu_{1} c \varphi_1-[\int_{\RR^N}k(y-x)\underline{v}^1(t,y;z,T)dy-\underline{v}^1(t,x;z,T)+f(x+z,\underline{v}^1)\underline{v}^1 (t,x;z,T)]\\
=&(\mu c -\lambda(\mu)) \varphi-(\mu_{1} c -\lambda(\mu_{1})) \varphi_1+f(x+z,0)\underline{v}^1(t,x;z,T)-f(x+z,\underline{v}^1)\underline{v}^1(t,x;z,T)\\
=&-(\mu_{1} c -\lambda(\mu_{1})) \varphi_1-f_{u}(x+z,y)(\varphi-\varphi_1)^{2}\quad\quad \big(\text{for some}\,\, y\in (0,M)\big)\\
\leq&-(\mu_{1} c -\lambda(\mu_{1})) \varphi_1-f_{u}(x+z,y)(\varphi)^{2}\\
=&[-(\mu_{1} c -\lambda(\mu_{1}))-f_{u}(x+z,y)\frac{(\varphi)^{2}}{\varphi_1}]\varphi_1\\
\leq&0.
\end{align*}
Hence \eqref{sub-solu-eq1} also  holds for $(t,x)\in\RR\times\RR^N$ with $\underline
v^1(t,x;z,T)> 0$. The proposition then follows.
\end{proof}

\begin{proposition}
\label{sub-solution-prop2} Let $\phi_0$ be the positive principal
eigenfunction of $\mathcal{K}-I+a_0(\cdot)I$ with $\|\phi_0\|_{X_p}=1$. Then for any $z\in\RR^N$ and
$0<b\ll1 $, $\underline v^2(t,x;z,b):=b \phi_0(x+z)$ is a sub-solution
of \eqref{main-shifted-eq}.
\end{proposition}

\begin{proof}
Fix $z\in\RR^N$.
Observe that
$$
\int_{\RR^N}k(y-x)\phi_0(y+z)dy-\phi_0(x+z)+f(x+z,0)\phi_0(x+z)=\lambda_0\phi_0(x+z)\quad \forall x\in\RR^N.
$$
Observe also that
$\ds\max_{x\in\RR^N}\lambda_0 \phi_0(x+z)>0$
and then
$$\lambda_0 b \phi_0(x+z)\geq (f(x+z,0)-f(x+z,b\phi_0(x+z)))b\phi_0(x+z)\quad \forall 0<b\ll 1.$$
It then follows that
$$
\int_{\RR^N}k(y-x)b\phi_0(y+z)dy-b\phi_0(x+z)+f(x+z,b\phi_0(x+z))b\phi_0(x+z)\geq 0\quad \forall x\in\RR^N,\, 0<b\ll 1.
$$
Hence $\underline v^2(t,x;z,b)$ is a sub-solution of \eqref{main-shifted-eq} for $0<b\ll 1$.
\end{proof}

For given $0<b\ll 1$, there is $M>0$ such that for
$(t,x)\in\RR\times\RR^N$ with $M-2\delta_0\leq x\cdot\xi+cT-ct\leq M$ ($\delta_0$ is the nonlocal dispersal distance
in \eqref{main-eq}),
\begin{equation}
\label{b-cond-eq1}
 \underline v^1(t,x;z,T)\geq b.
\end{equation}

\begin{proposition}
\label{sub-solution-prop3} Let $0<b\ll 1$ and $M>0$ be such that
\eqref{b-cond-eq1} holds and $z\in\RR^N$, $T>0$. Let
$$
{\underline u}(t,x;z,T,d_1,b)=\begin{cases}\max\{ b\phi_0(x+z),\underline
v^1(t,x;z,T,d_1)\}\quad {\rm for}\quad x\cdot\xi+cT-ct< M\cr \underline
v^1(t,x;z,T,d_1)\quad {\rm for}\quad x\cdot \xi+cT-ct\geq M.
\end{cases}
$$
Then ${\underline u}(t,x;z,T,d_1,b)$ is a sub-solution of
\eqref{main-shifted-eq}.
\end{proposition}

\begin{proof}
First, it is not difficult to see that for any $x,z\in\RR^N$, there are at most two
$t$s such that
$b\phi_0(x+z)=\underline v^1(t,x;z,T)$. Hence for any fixed $x,z\in\RR^N$, $\underline u(t,x;z,T)(:=\underline u(t,x;z,t,b,d_1))$ is continuous at every $t$
and is differentiable in $t$ for a.e. $t$. Moreover, for any $t$ at which $\underline u(t,x;z,T)$ is differentiable, there holds
$$
\frac{\p \underline u(t,x;z,T)}{\p t}\leq \int_{\RR^N} k(y-x)\underline u(t,y;z,T)dy -\underline u(t,x;z,T)+
\underline u(t,x;z,T) f(x+z,\underline u(t,x;z,T)).
$$
Therefore, $\underline u(t,x;z,T)$ is a sub-solution of \eqref{main-shifted-eq}.
\end{proof}

For given $d_2\geq 0$, let
$$
\bar v(t,x;z,T,d_2)=e^{-\mu(x\cdot\xi+cT-ct)}\phi(x+z)+d_2 e^{-\mu_1(x\cdot\xi+cT-ct)}\phi_1(x+z)
$$
and
$$
\bar u(t,x;z,T,d_2)=\min\{\bar v(t,x;z,T,d_2),u^{+}(x+z)\}.
$$
We may write $\bar v(t,x;z,T)$ and $\bar u(t,x;z,T)$ for $\bar v(t,x;z,T,d_2)$ and
$\bar u(t,x;z,T,d_2)$, respectively, if no confusion occurs.

\begin{proposition}
\label{sup-solution-prop1}
 For any $d_2\geq 0$,  $z\in\RR^N$,  and $T>0$,
  $\bar u(t,x;z,T)$ is a super-solution of \eqref{main-shifted-eq}.
\end{proposition}

\begin{proof} It suffices to prove that $\bar v(t,x;z,T)$ is a super-solution.\\
Let $\varphi_2=d_2 e^{-\mu_1(x\cdot\xi+cT-ct)}\phi_1(x+z)$.
 By direct calculation, we have
\begin{align*}
&\frac{\p \bar{v}}{\p t}- [\int_{\RR^N}k(y-x)\bar v(t,y;z,T)dy-\bar v(t,x;z,T)+f(x+z,\bar{v})\bar v(t,x;z,T)]\\
\geq &\frac{\p \bar v}{\p t}-[\int_{\RR^N}k(y-x)\bar v(t,y;z,T)dy-\bar v(t,x;z,T)+f(x+z,0)\bar v(t,x;z,T)]\\
=&(\mu_1 c -\lambda(\mu_1)) \varphi_2\\
\geq&0.
\end{align*}
The proposition thus follows.
\end{proof}

In the rest of this section, we fix $d_1^*\gg 1$, $d_2^*\geq 0$, and $0<b^*\ll 1$. Let
\begin{equation}
\label{u-plus-minus-eq1}
u_{0,z,T}^-(x)=\underbar u(0,x;z,T,d_1^*,b^*)
\quad {\rm and}\quad
u_{0,z,T}^+(x)=\bar u(0,x;z,T,d_2^*).
\end{equation}
Then by Proposition \ref{sub-solution-prop3},
\begin{align*}
u(t,x;u_{0,z,T}^-,z)&\geq \underbar u(t,x;z,T)\\
&=\underbar u(0,x;z,T-t)\\
 &=u_{0,z,T-t}^-(x).
\end{align*}
Similarly,
\begin{equation*}
u(t,x;u_{0,z,T}^+,z)\leq
u_{0,z,T-t}^+(x).
\end{equation*}

\begin{proposition}
\label{monotonicity-prop} For any given $z\in\RR^N$, the following hold:
\begin{itemize}
\item[(1)] For any $t_2>t_1>0$,
$$
u(t_2+t,x;u_{0,z,t_2}^-,z)\geq u(t_1+t,x;u_{0,z,t_1}^-,z)\,\,\,\, \forall
t>-t_1,\,\, x\in\RR^N;
$$

\item[(2)]
$$
u(t_2+t,x;u_{0,z,t_2}^+,z)\leq u(t_1+t,x;u_{0,z,t_1}^+,z)\,\,\,\,\forall t>-t_1,\,\, x\in\RR^N.
$$
\end{itemize}
\end{proposition}

\begin{proof}
(1) For given $z\in\RR^N$ and $t_2>t_1>0$, by Proposition
\ref{sub-solution-prop3},
\begin{align*}
u(t_2-t_1,x;u_{0,z,t_2}^-,z)&\geq \underline u(t_2-t_1,x;z,t_2)\\
&=u_{0,z,t_2-(t_2-t_1)}^-(x)\\
&=u_{0,z,t_1}^-(x).
\end{align*}
Hence
\begin{align*}
u(t_2+t,x;u_{0,z,t_2}^-,z)&=u(t_1+t,x;u(t_2-t_1,\cdot;u_{0,z,t_2}^-,z),z)\\
&\geq u(t_1+t,x;u_{0,z,t_1}^-,z).
\end{align*}
(1) is thus proved.

(2) It follows by the similar arguments in (1) and Proposition \ref{sup-solution-prop1}.

\end{proof}

\section{Existence of Traveling Wave Solutions and Proof of Theorem \ref{existence-thm}}
\label{existence-section}

In this section, we investigate the existence  of traveling wave solutions of \eqref{main-eq} and prove Theorem \ref{existence-thm}.
Throughout this section, we assume (H1)-(H4).


 Let $u^\pm_{0,z,T}$ be as in \eqref{u-plus-minus-eq1}.
Let
\begin{equation}
\label{phi-plus-minus-eq}
\Phi^\pm(x,z)=\lim_{\tau\to\infty} u(\tau,x;u_{0,z,\tau}^\pm,z)
\end{equation}
and
\begin{equation}
\label{u-plus-minus-eq}
U^\pm(t,x;z)=\lim_{\tau\to\infty} u(t+\tau,x;u_{0,z,\tau}^\pm,z).
\end{equation}
By Proposition \ref{monotonicity-prop}, the limits in the above exist for all $t\in\RR$ and $x,z\in\RR^N$.
Moreover, it is easy to see that $\Phi^-(x,z)$ is lower semi-continuous in $(x,z)\in\RR^N\times\RR^N$ and
$\Phi^+(x,z)$ is upper semi-continuous.

We will show that  $u=U^+(t,x;0)$ and $u=U^-(t,x;0)$ are traveling wave solutions of \eqref{main-eq}
in the direction of $\xi$ with  speed $c$ generated by $\Phi^+(\cdot,\cdot)$ and $\Phi^-(\cdot,\cdot)$, respectively,
 and that $\Phi(\cdot,\cdot):=\Phi^+(\cdot,\cdot)$ satisfies Theorem \ref{existence-thm}(1)-(2).

To this end, we first prove some lemmas.

\begin{lemma}
\label{existence-lm1}
For each $z\in\RR^N$,
$u(t,x)=U^\pm(t,x;z)$ are entire solutions of \eqref{main-shifted-eq}.
\end{lemma}

\begin{proof}
We prove the case that $u(t,x)=U^+(t,x;z)$. The other case can be proved similarly.

Fix $z\in\RR^N$.
Observe that for any $x\in\RR^N$,
\begin{align*}
&u(t+\tau,x;u_{0,z,\tau}^+,z)=u(\tau,x;u_{0,z,\tau}^+,z)+
\int_0^t \int_{\RR^N} k(y-x) u(s+\tau,y;u_{0,z,\tau}^+,z)dy ds\\
&+\int_0^ t\big[-u(s+\tau,x;u_{0,z,\tau}^+,z)+
u(s+\tau,x;u_{0,z,\tau}^+,z) f(x+z,u(s+\tau,x;u_{0,z,\tau}^+,z))\big ] ds
\end{align*}
Letting $\tau\to\infty$, we have
$$
u(t,x)=u(0,x)+\int_0^t\big[\int_{\RR^N} k(y-x)u(s,y)dy-u(s,x)+u(s,x)f(x+z,u(s,x))\big] ds.
$$
This implies that $u(t,x)$ is differentiable in $t$ and satisfies \eqref{main-shifted-eq} for all $t\in\RR$.
\end{proof}

Observe that
$$
U^\pm(t,x;z)=u(t,x; \Phi^\pm(\cdot,z),z)\,\,\forall t\in\RR,\,\, x,z\in\RR^N.
$$

\begin{lemma}
\label{existence-lm2}
 $u(t,x;\Phi^\pm(\cdot,z),z)=\Phi^\pm(x -ct\xi,z+ct\xi)$,
$\ds\lim_{x\cdot\xi\to -\infty}(\Phi^\pm(x,z)-u^+(x+z))=0$ and
$\ds\lim_{x\cdot\xi\to\infty}\frac{\Phi^\pm(x,z)}{ e^{-\mu x\cdot\xi}\phi(x+z)}=1$ uniformly in $z\in\RR^N$.
\end{lemma}

\begin{proof} We prove the lemma for $\Phi^+(\cdot,\cdot)$. It can be proved similarly for
$\Phi^-(\cdot,\cdot)$.

First of all, we have
\begin{align*}
u(t,x;\Phi^+(\cdot,z),z)&=\lim_{\tau\to\infty}u(t,x;u(\tau,x;u^+_{0,z,\tau},z),z)\\
&=\lim_{\tau\to\infty}u(t+\tau,x;u^+_{0,z,\tau},z)\\
&=\lim_{\tau\to\infty}u(t+\tau,x-ct\xi;u^+_{0,z+ct\xi,t+\tau},z+ct\xi)\\
&=\Phi^+(x  -ct\xi,z+ct\xi).
\end{align*}

Note that
\begin{align*}
\underbar {u}(t+T,x;z,T)&= e^{-\mu (x\cdot\xi -ct)}\phi(x+z)-d_1 e^{-\mu_1(x\cdot\xi -ct)}\phi_1(x+z) \\
&\leq u(t,x;\Phi^+(\cdot,z),z)\\
& \leq \bar{u}(t+T,x;z,T)\\
&=e^{-\mu(x \cdot \xi -ct)}\phi(x+z)+d_2  e^{-\mu_1(x\cdot\xi -ct)}\phi_1(x+z)
\end{align*}
 for $t\in\RR$ and $x,z\in\RR^N$. Thus $\ds\lim_{x \cdot \xi -ct \to \infty} \frac{\Phi^+(x  -ct\xi,z+ct\xi)}
 {e^{-\mu (x\cdot\xi -ct)}\phi(x+z)}=1$, which is equivalent to $\ds\lim_{x\cdot\xi\to\infty}
\frac{\Phi^+(x,z)}
{ e^{-\mu x\cdot\xi}\phi(x+z)}=1$, uniformly in $z\in\RR^N$.

We now prove that $\ds\lim_{x\cdot\xi\to -\infty}\big(\Phi^+(x,z)-u^+(x+z)\big)=0$ uniformly in $z\in\RR^N$. Observe that
there is $M>0$ such that
$$
U^+(t,x,z)\geq U^-(t,x,z)\geq b \phi_0(x+z) \quad {\rm for}\quad x\cdot\xi-ct\leq M,\,\, z\in\RR^N.
$$
By Proposition \ref{positive-solu-prop}, for any $\epsilon>0$, there are $T>0$ and $\eta^*\in\RR$ such that
$$
|U^+(T,x,z)-u^+(x+z)|<\epsilon\quad {\rm for}\quad x\cdot\xi\leq \eta^*,\,\, z\in\RR^N.
$$
This implies that
$$
|\Phi^+(x,z)-u^+(x+z)|\leq \epsilon\quad {\rm for}\quad x\cdot\xi\leq \eta^*+cT,\,\, z\in\RR^N
$$
and hence $\ds\lim_{x\cdot\xi\to -\infty}\big(\Phi^+(x,z)-u^+(x+z)\big)=0$ uniformly in $z\in\RR^N$.
\end{proof}

\begin{corollary}
\label{existence-cor1}
Both $\Phi^+(\cdot,\cdot)$ and $\Phi^-(\cdot,\cdot)$ generate traveling wave solutions of \eqref{main-eq}
in the direction of $\xi$
with speed $c$.
\end{corollary}

\begin{proof}
First of all, by Lemmas \ref{existence-lm1} and \ref{existence-lm2}, both  $\Phi^+(\cdot,\cdot)$ and $\Phi^-(\cdot,\cdot)$
satisfy \eqref{def-eq1} and \eqref{def-eq2}.

Next, for any $x,x^{'}\in\RR^N$ with $x\cdot\xi=x^{'}\cdot\xi$, $z\in\RR^N$, and $\tau\in\RR$,  we have
\begin{align*}
u(\tau,x^{'};u^\pm_{0,z-x^{'},\tau}(\cdot),z-x^{'})&=u(\tau,x;u^\pm_{0,z-x^{'},\tau}(\cdot+x^{'}-x),z-x^{'}+(x^{'}-x))\\
&=u(\tau,x;u^\pm_{0,z-x,\tau}(\cdot),z-x).
\end{align*}
This implies that $\Phi^\pm(\cdot,\cdot)$ satisfies \eqref{def-eq3}.

Observe now that $u^\pm_{0,z+p_i{\bf e_i},\tau}=u^\pm_{0,z,\tau}$ for any $\tau\in\RR$ and $z\in\RR^N$. It then follows
that $\Phi^\pm(x,z+p_i{\bf e_i})=\Phi^\pm(x,z)$ and hence $\Phi^\pm(\cdot,\cdot)$ satisfies \eqref{def-eq4}.

Therefore, both $\Phi^+(\cdot,\cdot)$ and $\Phi^-(\cdot,\cdot)$ generate traveling wave solutions of \eqref{main-eq}
in the direction of $\xi$
with speed $c$.
\end{proof}

\begin{lemma}
\label{existence-lm3}
$\ds\lim_{x\cdot\xi-ct\to -\infty} U^\pm_t(t,x;z)=0$
uniformly in $z\in\RR^N$.
\end{lemma}

\begin{proof}
Note that
\begin{align*}
U^\pm_t(t,x;z)&=\int_{\RR^N} k(y-x)U^\pm(t,y;z)dy-U^\pm(t,x;z)+U^\pm(t,x;z)f(x+z,U^\pm(t,x;z))\\
&=\int_{\|y\|\leq \delta_0} k(y)U^\pm(t,x+y;z)dy-U^\pm(t,x;z)+U^\pm(t,x;z)f(x+z,U^\pm(t,x;z)).
\end{align*}
Note also that
$$
\lim_{x\cdot\xi-ct\to -\infty} \big(U^\pm(t,x;z)-u^+(x+z)\big)=0
$$
uniformly in $z\in\RR^N$.
It then follows that
\begin{align*}
\lim_{x\cdot\xi-ct\to -\infty} U_t^\pm(t,x;z)&=\lim_{x\cdot\xi -ct\to -\infty} \Big[U^\pm_t(t,x;z)-\int_{\RR^N} k(y)u^+(y+x+z)dy+u^+(x+z)\\
&\quad -u^+(x+z)f(x+z,u^+(x+z))\Big]\\
&=\lim_{x\cdot\xi-ct\to -\infty}\Big[\int_{\RR^N} k(y)\big(U^\pm(t,x+y;z)-u^+(x+y+z)\big)dy\\
&\quad -\big(U^\pm(t,x;z)-u^+(x+z)\big)\\
&\quad + \big(U^\pm(t,x;z) f(x+z,U^\pm(t,x;z))-u^\pm(x+z) f(x+z,u^+(x+z))\big)\Big]\\
&=0\quad \text{uniformly in}\quad z\in\RR^N.
\end{align*}
\end{proof}

\begin{lemma}
\label{existence-lm4}
$\lim_{x\cdot\xi-ct\to\infty}\frac{U_t^\pm(t,x;z)}{e^{-\mu(x\cdot\xi-ct)}\phi(x+z)}=\mu c$ uniformly in $z\in\RR^N$.
\end{lemma}

\begin{proof}
We prove the lemma for $U^+(t,x;z)$. It can be proved similarly for $U^-(t,x;z)$.

First, let $U(t,x;z)=U^+(t,x;z)$.  By Lemma \ref{existence-lm2}, for any $\epsilon>0$, there is $M>0$ such that
for any  $x,z\in\RR^N$ and $t\in\RR$ with $x\cdot\xi-ct\geq M$,
\begin{equation}
\label{bd-eq1}
\big|\frac{U(t,x;z)}{e^{-\mu (x\cdot\xi-ct)}}-\phi(x+z)\big|<\epsilon
\end{equation}
and
\begin{equation}
\label{bd-eq2}
|f(x+z,U(t,x;z))-f(x+z,0)|<\epsilon.
\end{equation}
Observe that
\begin{equation}
\label{bd-eq3}
\mu c \phi(x+z)=\int_{\RR^N} e^{-\mu (y-x)\cdot\xi}k(y-x)\phi(y+z)dy-\phi(x+z)+a_0(x+z)\phi(x+z)
\end{equation}
for all $x,z\in\RR^N$,
where $a_0(x+z)=f(x+z,0)$, and
\begin{equation}
\label{bd-eq4}
U_t(t,x;z)=\int_{\RR^N} k(y-x)U(t,y;z)dy-U(t,x;z)+U(t,x;z)f(x+z,U(t,x;z))
\end{equation}
for all $t\in\RR$ and $x,z\in\RR^N$.
By \eqref{bd-eq1}-\eqref{bd-eq4}, we have
\begin{align*}
\big| \frac{U_t(t,x;z)}{e^{-\mu(x\cdot\xi-ct)}\phi(x+z)}-\mu c\big|&= \frac{1}{\phi(x+z)}\Big| \int_{\RR^N} e^{-\mu(y-x)\cdot\xi}k(y-x)
\big(\frac{U(t,y;z)}{e^{-\mu (y\cdot\xi-ct)}}-\phi(y+z)\big)
dy\\
&\quad -\big(\frac{U(t,x;z)}{e^{-\mu (x\cdot\xi -ct)}}-\phi(x+z)\big)\\
&\quad +\big(\frac{U(t,x;z)}{e^{-\mu (x\cdot\xi-ct)}}-\phi(x+z)\big) f(x+z,U(t,x;z))\\
&\quad +\phi(x+z)\big(f(x+z,U(t,x;z))- f(x+z,0)\big)\Big|\\
& \leq \epsilon \big[\int_{\RR^N} e^{-\mu(y-x)\cdot\xi}k(y-x)dy\\
&\quad +1+|f(x+z,U(t,x;z))|+\phi(x+z)\big]
\end{align*}
for all $x,z\in\RR^N$ and $t\in\RR$ with $x\cdot\xi-ct\geq M+\delta_0$, where $\delta_0$ is the nonlocal dispersal distance
in \eqref{main-eq}.
 It then follows
that
$$
\lim_{x\cdot\xi-ct\to\infty} \frac{U_t^\pm(t,x;z)}{e^{-\mu(x\cdot\xi-ct)}\phi(x+z)}=\mu c
$$
uniformly in $z\in\RR^N$.
\end{proof}

\begin{proof}[Proof of Theorem \ref{existence-thm}]
Let $\Phi(x,z)=\Phi^+(x,z)$ and $U(t,x;z)=U^+(t,x;z)$. Note that
$U(t,x;z)=u(t,x;\Phi(\cdot,z),z))$. We show that $\Phi(\cdot,\cdot)$ and $U(\cdot,\cdot;\cdot)$ satisfy
Theorem \ref{existence-thm}(1) and (2), respectively.

(1) It follows from Corollary \ref{existence-cor1} and Lemma \ref{existence-lm2}.

(2) By Lemmas \ref{existence-lm3} and \ref{existence-lm4},
we only need to prove that $U_t(t,x;z)>0$ for all $(t,x,z)\in\RR\times\RR^N\times\RR^N$.

For any $t_1<t_2$, we have
$$
u_{0,z,t_1}^+(x)\geq u_{0,z,t_2}^+(x)\quad \forall x,z\in\RR^N.
$$
Hence
\begin{align*}
u(t_1,x;\Phi^+(\cdot,z),z)&=u(t_2+t_1-t_2,x;\Phi^+(\cdot,z),z)\\
&=\lim_{n\to\infty} u(t_2,x;u(n+t_1-t_2,\cdot;u_{0,z,n}^+,z),z)\\
& \leq \lim_{n\to\infty}u(t_2,x;u(n+t_1-t_2,\cdot;u_{0,z,n+t_1-t_2}^+,z),z)\\
&=u(t_2,x;\Phi^+(\cdot,z),z).
\end{align*}
Therefore, $U(t,x;z)=u(t,x;\Phi^+(\cdot,z),z)$ is nondecreasing as $t$ increases.

Let $v(t,x;z)=u_t(t,x;\Phi^+(\cdot,z),z)$. Then $v(t,x;z)\geq 0$. By Lemma \ref{existence-lm4},
for any $t\in\RR$ and $z\in\RR^N$,
the set $\{x\in\RR^N\,|\,v(t,x;z)>0\}$ has positive Lebesgue measure. Note that $v(t,x;z)$ satisfies
\begin{equation}
\label{linear-eq1}
v_t(t,x;z)=\int_{\RR}k_{\delta}(y-x)v(t,y;z)dy-v(t,x;z)+a(t,x;z)v(t,x;z)
\end{equation}
where $a(t,x;z)=f(x+z,u(t,x;\Phi^+(\cdot,z),z))+u(t,x;\Phi^+(\cdot,z),z)f_u(x+z,u(t,x;\Phi^+(\cdot,z),z))$.
Then by Proposition \ref{comparison-prop}, we have
$$
v(t,x;z)>0\quad \forall t\in\RR,\, x,z\in\RR^N.
$$
This implies that  $U_t(t,x;z)>0$ for all $t\in\RR$ and $x,z\in\RR^N$.
\end{proof}

\section{Uniqueness and Continuity of Traveling Wave Solutions and Proof of Theorem \ref{uniqueness-thm}}
\label{uniqueness-section}

In this section, we investigate the uniqueness and continuity of traveling wave
solutions of \eqref{main-eq} and prove  Theorem \ref{uniqueness-thm} by the ``squeezing'' techniques developed in \cite{ChGu1} and  \cite{GuWu}.

Throughout this section, we fix $\xi\in S^{N-1}$ and $c>c^*(\xi)$. Let
$\mu^*$ be such that
$$
c^*(\xi)=\frac{\lambda_0(\mu^*,\xi,a_0)}{\mu^*}<\frac{\lambda_0(\tilde \mu,\xi,a_0)}{\tilde \mu}\quad \forall \tilde\mu \in(0,\mu^*).
$$
We fix $c>c^*(\xi)$  and $\mu\in(0,\mu^*)$ with $\frac{\lambda_0(\mu,\xi,a_0)}{\mu}=c$ and assume that  $U^\pm(t,x;z)$ and
$\Phi^\pm(x,z)$ are as in section \ref{existence-section}. We put $\Phi(x,z)=\Phi^+(x,z)$ and
$U(t,x;z)=U^+(t,x;z)$.
Let $U_1(t,x;z)=u(t,x;\Phi_1(\cdot,z),z)(\equiv\Phi_1(x-ct\xi,z+ct\xi))$.

We first prove some lemmas, some of which will also be used in next section.
By Lemmas \ref{existence-lm2} and  \ref{existence-lm4}, there is $M_0>0$ such that
\begin{equation}
\label{continuity-eq0}
0<\sup_{x\cdot\xi-ct\geq M_0,z\in\RR^N}\frac{U(t,x;z)}{U_t(t,x;z)}<\infty.
\end{equation}
 Observe that there is $\sigma_0>0$ such that
\begin{equation}
\label{continuity-eq1} U(t,x;z)\geq \sigma_0\quad {\rm for}\quad
x\cdot\xi-ct \leq M_{0}.
\end{equation}
Let
\begin{equation}
\label{continuity-eq2}
 \eta_0=\inf_{0<u\leq 2
u^+_{\sup}}(-f_u(x,u))\sigma_0,
\end{equation}
where $u_{\sup}^+=\sup_{x\in\RR^N}u^+(x)$.
Throughout the rest of this section, $M_0$, $\sigma_0$, $\eta_0$  are fixed and
satisfy \eqref{continuity-eq0}-\eqref{continuity-eq2}.

\begin{lemma}
\label{continuity-stability-lm1} Let $\epsilon_0\in (0,1)$ and $\eta\in
(0,(1-\epsilon_0)\eta_0)$. There is $l>0$ such that for each
$\epsilon\in (0,\epsilon_0)$,
  $$H^{\pm}(t,x;z)=(1\pm \epsilon e^{-\eta t})U(t \mp l \epsilon e^{-\eta t},x;z),\forall t\geq 0\,\, x,z\in \RR^N
  $$
 are super-/sub-solution of \eqref{main-shifted-eq}.
\end{lemma}

 \begin{proof}
First we prove that $H^{+}(t,x;z)$ is a super-solution of
\eqref{main-shifted-eq}. Let $h=\epsilon e^{-\eta t}$ and $\tau=t - l
\epsilon e^{-\eta t}$. Then
 $$H^{+}(t,x;z)=(1+h)U(\tau,x;z),\forall t\geq 0,\,\, x,z\in\RR^N.$$
 By direct calculation, we have
 \begin{align*}
 &\frac{\partial H^{+}(t,x;z)}{\partial t}-[\int_{\RR^N}k(y-x)H^{+}(t,y;z)dy-H^{+}(t,x;z)+H^{+}(t,x;z)f(x+z,H^{+}(t,x;z))]\\
 &=-\eta h U(\tau,x;z) + (1+l\eta h ) [(\mathcal{K}-I)H^++f(x+z,U)H^+]-[(\mathcal{K}-I)H^++f(x+z,H)H^+]\\
 &=-\eta h U(\tau,x;z) + l\eta h  [(\mathcal{K}-I)H^++f(x+z,U)H^+]+[f(x+z,U)-f(x+z,H)]H^+\\
 &=-\eta h U(\tau,x;z) + l\eta h  (1+h)U_t(\tau,x;z)+[f(x+z,U)-f(x+z,H^+)](1+h)U(\tau,x;z)\\
 &= h \eta U(\tau,x;z)[-1 +
 l(1+h)\frac{U_t(\tau,x+z)}{U(\tau,x+z)}-f_{u}(x+z,u^*(\tau,x;z))(1+h)U(\tau,x;z)/\eta],
 \end{align*}
 where $u^*(\tau,x;z)$ is some number between $U(\tau,x;z)$ and
 $H^+(t,x;z)$.
We only need to prove that
\begin{equation}
\label{continuity-eq3} -1 +
l(1+h)\frac{U_{\tau}(\tau,x;z)}{U(\tau,x;z)}-f_{u}(x+z,U^*(\tau,x;z))(1+h)U(\tau,x)/\eta\geq
0
\end{equation}
for all $t\geq 0$ and $x,z\in\RR^N$.

If $t\geq 0$ and $x\in\RR^N$ are such that $x\cdot\xi-c\tau\leq
M_0$, by \eqref{continuity-eq1}, \eqref{continuity-eq2}, and the fact that
$U_t(\tau,x,;z)>0$, \eqref{continuity-eq3} holds.

If $t\geq 0$ and $x\in\RR^N$ are such that $x\cdot\xi -c\tau\geq
M_0$, and $l\geq \sup_{x\cdot\xi-c\tau\geq
M_0}\frac{U(\tau,x;z)}{U_t(\tau,x;z)}$, then \eqref{continuity-eq3}
also holds.

By the similar arguments above, we can prove that $H^{-}(t,x;z)$ is a sub-solution of
\eqref{main-shifted-eq}. This completes the proof.

 \end{proof}

\begin{lemma}
\label{continuity-stability-lm2}
Let $\epsilon_0\in (0,1)$ be given and $l$ be as in Lemma \ref{continuity-stability-lm1}. For any given $0<\epsilon_1\leq \epsilon_0$,
 there exists constant  $M_1(\epsilon_1)>0$ such that for all $\epsilon \in (0,\epsilon_1]$
 $$
 (1-\epsilon)U(t+3l\epsilon,x;z) \leq U(t,x;z) \leq (1+\epsilon)U(t-3l\epsilon,x;z)\,\, \forall t\in\RR,\,\, x,z\in\RR^N,\,\,   x-ct \leq -M_1(\epsilon_1).
 $$
\end{lemma}

\begin{proof}
Let $h(s)=(1+s)U(t-3ls,x;z)$. Then,
$h'(s)=U(t-3 l s,x;z)-3 l U_t(t-3 ls,x;z)$. By Lemma \ref{existence-lm3}, there exists a $M(\epsilon_1)>0$ such that
$h'(s)>0$ for  $s \in [-\epsilon_1,\epsilon_1]$, $x-ct \leq -M_{1}(\epsilon_1)$, and $z\in\RR^N$. Hence, the lemma follows.
\end{proof}

\begin{lemma}
\label{continuity-lm2}
 For any $\epsilon > 0$, there exists a constant $C(\epsilon)\geq 1$ such that
$$
U_1(t-2\epsilon,x;z)\leq U(t,x;z)\leq U_1(t+2\epsilon,x;z)\quad
\forall t\in\RR,\,\, x,z\in\RR^N,\,\, x \cdot \xi-ct \geq C(\epsilon).
$$
\end{lemma}
 \begin{proof}
 It follows from the fact that
 \begin{align*}
 \lim_{x\cdot\xi-ct \to\infty} \frac{U_1(t,x;z)}{e^{-\mu(x\cdot\xi- ct)}\phi(x+z)}&=\lim_{x\cdot\xi-ct\to\infty}
 \frac{U_1(t,x;z)}{U(t,x;z)}\frac{U(t,x;z)}{e^{-\mu(x\cdot\xi-ct)}\phi(x+z)}\\
 &=\lim_{x\cdot\xi-ct\to\infty}\frac{U(t,x;z)}{e^{-\mu(x\cdot\xi-ct)}\phi(x+z)}\\
 &
 =1
 \end{align*}
 uniformly in $z\in\RR^N$.
 \end{proof}

\begin{lemma}
\label{continuity-lm3}
Let $\epsilon_0\in (0,1)$ and $\eta_0$, $l$ be as in Lemma \ref{continuity-stability-lm1}.
 For any given  $\epsilon \in(0,\epsilon_0)$, there is  $\tau>0$  such that
$$
(1-\epsilon e^{-\eta t})U(t-\tau+l\epsilon e^{-\eta t},x)\leq U_1(t,x;z)\leq (1+\epsilon e^{-\eta t})U(t+\tau-l\epsilon e^{-\eta t},x;z)
$$
for all  $x,z \in \RR^N$ and $t\geq 0$.
\end{lemma}
\begin{proof} First by Propositions \ref{positive-solu-prop} and
\ref{comparison-prop},
$$
0<U(t,x;z)<u^+(x+z)\quad {\rm and}\quad 0<U_1(t,x;z)<u^+(x+z)\quad \forall t\in\RR,\,\, x,z\in\RR^N.
$$
Then by Lemma \ref{continuity-lm2}, there exists a constant $C(1)$ such that
$$
U_1(t,x;z)\geq U(t-2,x;z)\quad \forall t\in\RR\,\, x,z\in\RR^N,\,\, x\cdot\xi-ct\geq C(1).
$$
By \eqref{def-eq2}, there is $t_1\geq 2$ such that
$$
U_1(t,x;z)\geq (1-\epsilon) U(t-t_1,x;z)\quad \forall\,\, t\in\RR,\,\, x,z\in\RR^N,\,\, x\cdot\xi-ct<C(1).
$$
Thus
$$U_1(0,x;z)\geq (1-\epsilon)U(-t_1,x;z)=(1-\epsilon)U(-(t_1+l\epsilon )+l\epsilon,x;z)\quad\forall\, x,z\in\RR^N.
$$
It then follows  Lemma \ref{continuity-stability-lm1} that
$$
U_1(t,x;z)\geq (1-\epsilon e^{-\eta t})U(t-(t_1+l\epsilon)+l\epsilon e^{-\eta t},x;z)\quad \forall t\geq 0, x,z \in \RR^N.
$$

Similarly, it can be proved that there is $t_2\geq 2$ such that
$$
U_1(t,x;z)\leq (1+\epsilon e^{-\eta t})U(t+t_2+l\epsilon-l\epsilon e^{-\eta t},x;z) \quad \forall t\geq 0,\,\, x,z\in\RR^N.
$$

The lemma then follows with $\tau=\max\{t_1+l\epsilon, t_2+l \epsilon\}$.
\end{proof}

\begin{lemma}
\label{continuity-lm5}
Let $\tau>0, t_{1}>0,$ and $M \in \RR$ be given. Suppose that $W^\pm(t,x;t_1,z)$ are the solution of \eqref{main-shifted-eq}
with initial
 $$
 W^\pm(0,x;t_1,z)= U(t_{1}\pm\tau,x;z)\varsigma(x-ct_{1}-M)+U(t_{1}\pm2\tau,x;z)(1-\varsigma(x-ct_{1}-M)),
 $$
  where $\varsigma(s)=0$ for $s \leq 0$ and $\varsigma(s)=1$ for $s > 0$.
   Then
   $$
  W^+(1,x;t_1,z)\leq   (1+\epsilon)U(t_{1}+1+2\tau-3l\epsilon,x;z)
   $$
   and
\begin{equation*}
 W^-(1,x;t_1,z)\geq  (1-\epsilon)U(t_{1}+1-2\tau+3l\epsilon,x;z)
 \end{equation*}
for all  $x,z\in\RR^N$ with $x-c(1+t_{1}) \leq M$ provided that $0<\epsilon\ll 1$.
\end{lemma}

\begin{proof} We give a proof for $W^-(1,x;t_1,z)$. The case of $W^+$ can be proved similarly.
Note that
$$
W^-(0,x;t_1,z)\geq  U(t_1-2\tau,x;z)\quad \forall x,z\in\RR^N.
$$
 It then follows that
$$
W^-(1,x;t_1,z)>U(1+t_1-2\tau,x;z)\quad\forall x,z\in\RR^N.
$$

Take an  $\epsilon_1\in (0,\epsilon_0]$. By Lemma \ref{continuity-stability-lm2}, for any $\epsilon\in (0,\epsilon_1]$,
$$
W^-(1,x;t_1,z)>(1-\epsilon)
U(1+t_1-2\tau+3l\epsilon,x;z)\quad \forall
 x\cdot\xi
-c(t_1+1)\leq -M(\epsilon_1),\,\, z\in\RR^N.
$$

We claim that for $0< \epsilon\ll 1$,
$$
W^-(1,x;t_1,z)>(1-\epsilon)
U(1+t_1-2\tau+3l\epsilon,x;z)
\quad \forall x\cdot\xi-c(t_1+1)\in [-M(\epsilon_1),
M],\,z\in\RR^N.
$$
In fact,
let $W(t,x;z)=W^-(t,x;t_1,z)-U^+(t+t_1-2\tau,x;z)$ and
\begin{align*}
h=\inf_{t\in [0,1],x,z\in\RR^N}&\{[ W^-(t,x;t_1,z)f(x+z,u(t,x;u_{0,z},z))\\
&
-U(t+t_1-2\tau,x;z)f(x+z,U(t+t_1-2\tau,x;z))]\\
&\cdot\frac{1}{ W^-(t,x;t_1,z)-U(t+t_1-2\tau,x;z)}\}.
\end{align*}
Then
$$
W(0,x;z)=\begin{cases} U(t_1-\tau,x;z)-U(t_1-2\tau,x;z)\quad {\rm for}\quad x\cdot\xi -ct_1> M\cr
0\quad {\rm for}\quad x\cdot\xi -ct_1\leq M
\end{cases}
$$
and
$$
W_t(t,x;z)\geq \int_{\RR^N} k(y-x) W(t,y;z)dy-W(t,x;z)+h W(t,x;z)\quad \forall t\in [0,1],\,\, x,z\in\RR^N.
$$
It then follows  that
$$
W(1,\cdot;z)\geq e^{-1+h} (W(0,\cdot;z)+\mathcal{K} W(0,\cdot;z)+\frac{\mathcal{K}^2}{2!} W(0,\cdot;z)+\cdots),
$$
where $\mathcal{K}W(0,\cdot;z)$ is defined as in \eqref{k-delta-op} with $u$ being replaced by $W(0,\cdot;z)$.
By Lemma \ref{existence-lm2}, there are
$\tilde\sigma>0$ and $\tilde M>0$ such that
\begin{equation}
\label{stab-eq3}
U(t_1-\tau,x;z)-U(t_1-2\tau,x;z)\geq\tilde \sigma\quad \forall \, x,z\in\RR^N\,\,\,{\rm with}\,\,\, \tilde M\leq  x\cdot\xi-ct_1\leq \tilde M+1.
\end{equation}
This implies that
\begin{equation}
\label{stab-eq4}
W(1,x;z)\geq U(1+t_1-2\tau+3l\epsilon,x;z)-U(1+t_1-2\tau,x;z)
\end{equation}
for $x\cdot\xi-c(t_1+1)\in [-M(\epsilon_1),
M]$ and $z\in\RR^N$ provided that $0<\epsilon \ll 1$.
By \eqref{stab-eq3} and \eqref{stab-eq4}, we have
\begin{align*}
W^-(1,x;t_1,z)
&=W(1,x;z)+U(1+t_1-2\tau,x;z)\\
&\geq U(1+t_1-2\tau+3l\epsilon,x;z)\\
&\geq (1-\epsilon) U(1+t_1-2\tau+3l\epsilon,x;z)
\end{align*}
for $x\cdot\xi-c (1+t_1)\leq M$ and $z\in\RR^N$ provided that    $0<\epsilon\ll 1$.

\end{proof}

\begin{proof}[Proof of Theorem \ref{uniqueness-thm}]
(1)
Let
$$
A^+=\{\tau\geq 0\,|\, \limsup_{t\to\infty}\sup_{x,z\in\RR^N}\frac{U_1(t,x;z)}{U(t+2\tau,x;z)}\leq 1\}
$$
and
$$
A^-=\{\tau\geq 0\,|\, \liminf_{t\to\infty}\inf_{x,z\in\RR^N} \frac{U_1(t,x;z)}{U(t-2\tau,x;z)}\geq 1\}.
$$
By Lemma \ref{continuity-lm3}, $A^\pm\not =\emptyset$.
Let
$$
\tau^+=\inf\{\tau\,|\,\tau\in A^+\},\quad
\tau^-=\inf\{\tau\,|\,\tau\in A^-\}.
$$

We first claim that
$\tau^\pm\in A^\pm$.
 In fact, let $\tau_n\in A^+$ be such that $\tau_n\to \tau^+$.  Then
for any $0<\epsilon<1$,
there are $t_n\to \infty$  such that
$$
\frac{U_1(t,x;z)}{U(t+2\tau_n,x;z)}\leq 1+\epsilon\quad \forall x,z\in \RR^N,\,\, t\geq t_n
$$
and
$$
\frac{U(t+2\tau^+,x;z)-U(t+2\tau_n,x;z)}{U(t+2\tau_n,x;z)}>-\epsilon\quad \forall n\gg 1, \, t\in\RR,\, x,z\in\RR^N.
$$
Observe that
\begin{equation*}
\frac{U_1(t,x;z)}{U(t+2\tau^+,x;z)}=\frac{U_1(t,x;z)}{U(t+2\tau_n,x;z)}\frac{U(t+2\tau_n,x;z)}{U(t+2\tau^+,x;z)}
\end{equation*}
and
\begin{align*}
\frac{U(t+2\tau_n,x;z)}{U(t+2\tau^+,x;z)}&=
\frac{1}{1+\frac{U(t+2\tau^+,x;z)-U(t+2\tau_n,x;z)}{U(t+2\tau_n,x;z)}}\\
&\leq \frac{1}{1-\epsilon}\\
&\leq 1+\epsilon\quad \forall n\gg 1.
\end{align*}
Fix $n\gg 1$. Then
$$
\sup_{x,z\in\RR^N}\frac{U_1(t,x;z)}{U(t-2\tau^-,x;z)}\leq (1+ \epsilon)^2\quad \forall t\geq t_n.
$$
This implies that
$\tau^+\in A^+$. Similarly, we have $\tau^-\in A^-$.

Next we claim that
$\tau^\pm=0$.  Assume that $\tau^->0$. Note that
$$
\liminf_{t\to\infty}\inf_{x,z\in\RR^N}\frac{U_1(t,x;z)}{U(t-2\tau^-,x;z)}\geq 1.
$$
Hence for any $\bar \epsilon>0$, there is $t_0>0$ such that
$$
\frac{U_1(t_0,x;z)}{U(t_0-2\tau^-,x;z)}\geq 1- \bar \epsilon\quad \forall x,z\in\RR^N.
$$
This implies that
$$
U_1(t_0,x;z)\geq (1- \bar \epsilon) U(t_0-2\tau^-,x;z)\geq U^+(t_0-2\tau^-,x;z)-\hat \epsilon
$$
where $\hat \epsilon=\bar \epsilon\max_{t,x,z}U^+(t,x,z)$. By Lemma \ref{continuity-lm2},
for $x\cdot\xi- ct_0\geq M:= C(\tau^-/2)$,
$$
U_1(t_0,x;z)\geq U(t_0-\tau^-,x;z).
$$
This implies that
$$
U_1(t_0,x;z)\geq U(t_0-2\tau^-,x;z)(1-\zeta(x\cdot\xi -ct_0-M))+U(t_0-\tau^-,x;z)\zeta(x\cdot\xi-c t_0-M)-\hat \epsilon.
$$
Note that there is $K>0$ such that $U_1(t,x;z)+\hat\epsilon e^{Kt}$ is a super-solution of
\eqref{main-shifted-eq} for $t\in [0,1]$ provided that $0<\hat\epsilon\ll 1$.
By Lemma \ref{continuity-lm5}, for $0<\bar \epsilon\ll 1$ and $0<\epsilon\ll 1$,
$$
U_1(t_0+1,x;z)+\hat\epsilon e^K\geq (1-\epsilon)U(t_0+1-2\tau^-+3l\epsilon,x;z) \quad \forall x\cdot\xi-c(t_0+1)\leq M,\, z\in\RR^N,
$$
where $l$ is as in Lemma \ref{continuity-stability-lm1}.
Then for $0<\bar \epsilon\ll \epsilon\ll 1$,
$$
U_1(t_0+1,x;z)\geq (1-2\epsilon) U(t_0+1-2 z^-+3l\epsilon,x;z)\quad\forall  x\cdot\xi-c(t_0+1)\leq M,\, z\in\RR^N.
$$
By Lemma \ref{continuity-lm2} again, for $x\cdot\xi-c (t_0+1)\geq M$ , $z\in\RR^N$, and $0<\epsilon\ll 1$,
\begin{align*}
U_1(t_0+1,x;z)&>U(t_0+1-\tau^-,x;z)\\
&\geq (1-2\epsilon) U(t_0+1-\tau^-,x;z)\\
&\geq (1-2\epsilon) U(t_0+1-2\tau^-+3l\epsilon,x;z).
\end{align*}
Therefore for $0<\epsilon\ll 1$,
$$
U_1(t_0+1,x;z)\geq (1-2\epsilon) U(t_0+1-2\tau^-+3l \epsilon,x;z)\quad \forall x,z\in\RR^N.
$$
By Lemma \ref{continuity-stability-lm1},
$$
U_1(t_0+t+1,x;z)\geq (1-2\epsilon e^{-\tau t}) U(t_0+1+t-2\tau^-+2l\epsilon e^{-\eta t}+l\epsilon,x;z)\quad\forall t\geq 0,\, x,z\in\RR^N.
$$
It then follows that
$$
\tau^--\frac{l\epsilon}{2}\in A^-.
$$
this is a contradiction. Therefore $\tau^-=0$.
Similarly, we have $\tau^+=0$.

We now prove that $\Phi_1(x,z)=\Phi(x,z)$.
Recall that $U_1(t,x;z)=\Phi_1(x-ct\xi,z+ct\xi)$ and $U(t,x;z)=\Phi(x-ct\xi,z+ct\xi)$. Hence
\begin{align*}
\inf_{x,z\in\RR^N} \frac{U_1(t,x;z)}{U(t,x;z)}&=\inf_{x,z\in\RR^N}\frac{\Phi_1(x-ct\xi,z+ct\xi)}{\Phi(x-ct\xi,z+ct\xi)}\\
&=\inf_{x,z\in\RR^N}\frac{\Phi_1(x,z)}{\Phi(x,z)}
\end{align*}
and
\begin{align*}
\sup_{x,z\in\RR^N} \frac{U_1(t,x;z)}{U(t,x;z)}&=\sup_{x,z\in\RR^N}\frac{\Phi_1(x-ct\xi,z+ct\xi)}{\Phi(x-ct\xi,z+ct\xi)}\\
&=\sup_{x,z\in\RR^N}\frac{\Phi_1(x,z)}{\Phi(x,z)}
\end{align*}
This together with $\tau^\pm=0$ implies that
$$
\inf_{x,z\in\RR^N}\frac{\Phi_1(x,z)}{\Phi(x,z)}=\sup_{x,z\in\RR^N}\frac{\Phi_1(x,z)}{\Phi(x,z)}=1.
$$
We then must have
$\Phi_1(x,z)\equiv \Phi(x,z)$.

(2) Let $\Phi_1(x,z)=\Phi^-(x,z)(=U^-(0,x;z))$. By (1), $\Phi^-(x,z)=\Phi(x,z)$.
Recall that $\Phi^-(x,z)$ is lower semi-continuous and $\Phi^+(x,z)$ is upper semi-continuous.
We then must have that $\Phi(x,z)$ is continuous in $(x,z)\in\RR^N\times\RR^N$.

\end{proof}

\begin{corollary}
\label{continuity-cor1}
Let $\Phi(x,z)$ be as above. Then
$$
\lim_{\tau\to\infty}u(\tau,x;\underline u(0,\cdot;z,\tau,d_1,b),z)=\lim_{\tau\to\infty}u(\tau,x;\bar u(0,\cdot;z,\tau,d_2),z)=\Phi(x,z)
$$
for all $d_1\gg 1$, $d_2>0$, $0<b\ll 1$, and $x,z\in\RR^N$.
\end{corollary}

\begin{proof}
By the arguments of Theorem \ref{existence-thm}(3), for any $d_1\gg 1$ and $0<b\ll 1$,
$$
\lim_{\tau\to\infty}u(\tau,x;\underline u(0,\cdot;z,\tau,d_1,b),z)=\Phi^+(x,z)(=\Phi(x,z))\quad \forall x,z\in\RR^N,
$$
and for any $d_2\gg 1$,
$$
\lim_{\tau\to\infty}u(\tau,x;\bar u(0,\cdot;z,\tau,d_2),z)=\Phi^-(x,z)(=\Phi(x,z))\quad \forall x,z\in\RR^N.
$$
The corollary then follows.
\end{proof}

\section{Stability of Traveling Wave Solutions and Proof of Theorem \ref{stability-thm}}
\label{stability-section}

In this section, we investigate the stability  of traveling wave solutions of
\eqref{main-eq} and prove Theorem \ref{stability-thm}.

Throughout this section, we fix $\xi\in S^{N-1}$ and $c>c^*(\xi)$. Let
$\mu^*$ be such that
$$
c^*(\xi)=\frac{\lambda_0(\mu^*,\xi,a_0)}{\mu^*}<\frac{\lambda_0(\tilde \mu,\xi,a_0)}{\tilde \mu}\quad \forall \tilde\mu \in(0,\mu^*).
$$
We fix $c>c^*(\xi)$  and $\mu\in(0,\mu^*)$ with $\frac{\lambda_0(\mu,\xi,a_0)}{\mu}=c$. Let
$U(t,x;z)=U^+(t,x;z)$, where $U^+(t,x;z)$ is as in  section \ref{existence-section}.
We put
$u(t,x)=u(t,x;u_0,0)$, where $u_0$ is as in Theorem \ref{stability-thm}, and  put
$U(t,x)=U^+(t,x;0)$. We can prove Theorem \ref{stability-thm} by Lemmas \ref{stability-lm1}-\ref{stability-lm3} and the similar arguments in Theorem \ref{uniqueness-thm}. Here we only state these lemmas without proofs, which can be proved by properly modifying the arguments in their counterparts of Lemmas \ref{continuity-lm2}-\ref{continuity-lm5}.

\begin{lemma}
\label{stability-lm1}
For any $\epsilon > 0$,
there exists a constant $C_0(\epsilon)\ge 1$ such that
 $$
 u(t-2\epsilon,x) \leq U(t,x) \leq u(t+2\epsilon,x)\quad \forall x \cdot \xi-ct \geq C_0(\epsilon),\,\,t\geq 2\epsilon.
 $$
\end{lemma}

\begin{lemma}
\label{stability-lm2}
Let $\epsilon_0$, $\eta$, and $l$ be as in Lemma \ref{continuity-stability-lm1}.
For given $\epsilon\in (0,\epsilon_0)$, there are $t_\pm>0$ and $\tau_\pm>0$ such that
 $$
 (1-\epsilon e^{-\eta (t-t_-)})U(t-\tau_{-}+l\epsilon e^{-\eta (t-t_-)},x) \leq u(t,x) \leq (1+\epsilon e^{-\eta (t-t_+)})U(t+\tau_{+}
 -l\epsilon e^{-\eta (t-t_+)},x)
 $$
 for all
 $ x \in \RR^N$ and $t \geq \max\{t_{-},t_+\}$.
 \end{lemma}

%
%

\begin{lemma}
\label{stability-lm3}
Let $\tau>0$, $t_{1}>0,$ and $M \in \RR$ be given. Suppose that $w^{\pm}(\cdot,x;t_{1})$ are the solution of \eqref{main-eq} for $t \geq 0$ with the initial
conditions
 $$
 w^{\pm}(0,x;t_{1})=U(t_{1}\pm \tau,x)\varsigma(x-ct_{1}-M)+U(t_{1}\pm 2\tau,x)(1-\varsigma(x-ct_{1}-M)) \quad \forall x\in \RR^N,
 $$
  where $\varsigma(s)=0$ for $s \leq 0$ and $\varsigma(s)=1$ for $s > 0$. Then
\begin{align*}
 w^{+}(1,x;t_{1})\leq (1+\epsilon)U(t_{1}+1+2\tau-3l\epsilon)\\
 w^{-}(1,x;t_{1})\geq (1-\epsilon)U(t_{1}+1-2\tau+3l\epsilon),
\end{align*}
for all $ x\cdot\xi-ct_{1} \leq M+c$ and $0<\epsilon\ll 1$.
\end{lemma}

\end{document}